\documentclass[a4paper,12pt]{amsart}

\newcommand{\diag}{\text {\rm diag}}
\newcommand{\Lie}{\text {\rm Lie}}
\newcommand{\Ad}{\text {\rm Ad}}
\newcommand{\car}{\mathcal R}

\def\a{\alpha}

\def\c{\chi}

\def\d{\delta}
\def\D{\Delta}

\def\th{\theta}

\def\i{^{-1}}

\def\cc{\mathcal C}

\def\cf{\mathcal F}
\def\cg{\mathcal G}

\def\ci{\mathcal I}

\def\cl{\mathcal L}
\def\cm{\mathcal M}

\def\co{\mathcal O}

\def\car{\mathcal R}

\usepackage{enumerate}
\usepackage{amssymb} 
\usepackage{latexsym} 
\usepackage{amsfonts} 
\usepackage{amsmath} 
\usepackage{eucal} 
\usepackage{bm} 
\usepackage{bbm} 
\usepackage{graphicx} 
\usepackage[english]{varioref} 
\usepackage[nice]{nicefrac} 
\usepackage[all]{xy}

\newcommand{\End}{{\rm End}}
\newcommand{\Hom}{{\rm Hom}}

\newcommand{\Spec}{\textrm{Spec}}

\usepackage{amsthm}

\theoremstyle{plain}
\newtheorem{thm}{Theorem}[section] 
\newtheorem*{thm*}{Theorem} 
 \newtheorem{prop}[thm]{Proposition}
 \newtheorem{lem}[thm]{Lemma}
 \newtheorem{cor}[thm]{Corollary}
 \newtheorem{remark}[thm]{Remark}

\theoremstyle{definition}

\newtheorem{defn}[thm]{Definition}

\theoremstyle{remark}

\newtheorem*{remark*}{Remark}
\newtheorem*{claim*}{Claim}

\begin{document}

\author{Xuhua He}
\address{Department of Mathematics, Stony Brook University, Stony Brook,
NY 11794, USA}
\email{hugo@math.sunysb.edu}
\author{Jesper Funch Thomsen}
\address{Institut for matematiske fag\\ Aarhus Universitet\\ 8000 \AA rhus C,
Denmark} \email{funch@imf.au.dk}

\title[]{Frobenius splitting and geometry of $G$-Schubert varieties}

\begin{abstract}
Let $X$ be an equivariant embedding of a connected reductive group
$G$ over an algebraically closed field $k$ of positive
characteristic. Let $B$ denote a Borel subgroup of $G$. A
$G$-Schubert variety in $X$ is a subvariety of the form $\diag(G)
\cdot V$, where $V$ is a $B \times B$-orbit closure in $X$. In the
case where $X$ is the wonderful compactification of a group of
adjoint type, the $G$-Schubert varieties are the closures of
Lusztig's $G$-stable pieces. We prove that $X$ admits a Frobenius
splitting which is compatible with all $G$-Schubert varieties.
Moreover, when $X$ is smooth, projective and toroidal, then any $G$-Schubert 
variety in $X$ admits a stable Frobenius splitting along an ample 
divisors. Although this
indicates that $G$-Schubert varieties have nice singularities we
present an example of a non-normal $G$-Schubert variety in the 
wonderful compactification of a group of type $G_2$.
Finally we also extend the
Frobenius splitting results to the more general class of
$\mathcal R$-Schubert varieties.
\end{abstract}

\maketitle

\section{Introduction}

Let $G$ denote a connected and reductive group over
an algebraically closed field $k$,
and let $B$ denote a Borel subgroup of $G$.
An equivariant embedding $X$ of $G$ is a $G \times G$-variety
which contains $G=(G \times G)/\diag(G)$ as an open $G \times G$-invariant subset, where $\diag(G)$ is the diagonal image of $G$ in $G \times G$. 
Any equivariant embedding $X$ of $G$ contains finitely many $B
\times B$-orbits. In recent years the geometry of closures of $B
\times B$-orbits has been studied by several authors. The most
general result was obtained in \cite{HT2} where it was proved that
$B \times B$-orbit closures are normal, Cohen-Macaulay and have
($F$-)rational singularities (actually, even stronger results were
obtained). In the present paper we will study (closed) subvarieties 
in $X$ of the form $\rm{diag}(G) \cdot V$, where $V$ denotes the 
closure of a $B \times B$-orbit. Subvarieties of equivariant embeddings
of $G$ of this form will be called $G$-Schubert varieties. 

When $G$ is a semisimple group of adjoint type there exists a
canonical equivariant embedding ${\bf X}$ of $G$ which is 
called the {\it wonderful compactification}. The wonderful 
compactifications are of primary interest in this paper. 
Actually, this work arose from the question of describing 
the closures of the so-called $G$-stable pieces of ${\bf X}$. 
The $G$-stable pieces makes up a decomposition of ${\bf X}$ 
into locally closed subsets. They were introduced by Lusztig 
in \cite{L} where they were  used to construct
and study a class of perverse sheaves which generalizes his theory
of character sheaves on reductive groups. More precisely, these
perverse sheaves are the intermediate extensions of the so-called ``character sheaves'' on a $G$-stable piece. This motivates the study of  
closures of $G$-stable pieces which turns out to coincide 
with the set of $G$-Schubert varieties. 

Before discussing the closures of $G$-stable pieces in details, 
let us make a short digression and discuss some
other motivations for studying $G$-stable pieces 
and $G$-Schubert varieties (in wonderful 
compactifications):

\begin{enumerate}
\item When $G$ is a simple group, the boundary of the 
closure of the unipotent subvariety of $G$ in the 
wonderful compactification ${\bf X}$, is a union of certain 
$G$-Schubert varieties (see \cite{He1} and
\cite{HT1}). Thus knowing the geometry of these 
$G$-Schubert varieties will help us to understand 
the geometry of the closure of the
unipotent variety within ${\bf X}$.

\item Let $\Lie(G)$
denote the Lie algebra of a simple group $G$ over a field
of characteristic zero. Let $\ll , \gg$ denote a fixed
symmetric non-degenerate ad-invariant bilinear form. Let $< , >$ be
the bilinear form on $\Lie (G) \oplus \Lie(G)$ defined by
 $$<(x, y),
(x', y')>=\ll x, x' \gg-\ll y, y' \gg.$$ 
In \cite{EL}, Evens and Lu
showed that each splitting $\Lie(G) \oplus \Lie(G)=l \oplus l'$,
where $l$ and $l'$ are Lagrangian subalgebras of $\Lie(G) \oplus
\Lie(G)$, gives rise to a Poisson structure $\Pi_{l, l'}$ on ${\bf
X}$. If moreover, one starts with the Belavin-Drinfeld splitting,
then all the $G$-stable pieces/$G$-Schubert varieties and $B \times B^-$-orbits of ${\bf
X}$ are Poisson subvarieties, where $B^-$ is a Borel subgroup
opposite to $B$. Thus to understand the Poisson structure on ${\bf
X}$ corresponding to the Belavin-Drinfeld splitting, one needs to
understand the geometry of the $G$-stable pieces/$G$-Schubert varieties.
If we start with another splitting, then we obtain a
different Poisson structure on ${\bf X}$ and in order to understand
these Poisson structures, one needs to study the $\mathcal R$-stable pieces
\cite{LY} instead (see Section \ref{Rstable}), which generalize both the $G$-stable pieces and
the $B \times B^-$-orbits.
\end{enumerate}

The main technical ingredient in this paper is the 
positive characteristic notion
of Frobenius splitting. Frobenius splitting is a powerful
tool which has been proved to be very useful in obtaining 
strong geometric conclusions for e.g. Schubert varieties and 
closures of $B \times B$-orbits in equivariant embeddings.
In the present paper we obtain two types of results related
to $G$-Schubert varieties over fields of positive characteristic. 
First of all, if we fix an equivariant
embedding $X$ of a reductive group $G$ then we prove that 
all $G$-Schubert varieties in $X$ are simultaneously 
compatibly Frobenius split by a Frobenius splitting of $X$. 
Secondly, concentrating on a single $G$-Schubert variety 
$\mathcal X$, in a smooth projective and toroidal embedding $X$, we prove that 
this admits a stable Frobenius splitting along an ample 
divisor. Statements of this form
put strong conditions on the intertwined behavior of 
cohomology groups of line bundles on $X$ and its $G$-Schubert 
varieties. As this is related to geometric properties it
therefore seems natural to expect that $G$-Schubert 
varieties should have nice singularities. 
It therefore comes as a complete surprise that $G$-Schubert 
varieties, in general, are not even normal. We only provide 
a single example of this phenomenon (in the wonderful 
compactification of a group of type $G_2$), but expect 
that this absence of normality is the general picture.

In obtaining the Frobenius splitting result mentioned 
above, we have developed some general theory of how 
to construct Frobenius splitting of varieties of the
form $G \times_P X$ (see Section \ref{Equivalence}
for the definition). This part of the paper is 
influenced by the theory of $B$-canonical Frobenius
splitting as discussed in \cite[Chap.4]{BK}; 
in particular the proof of \cite[Prop.4.1.17]{BK}.
The presentation we provide is more general and 
makes it possible to extract even better result 
from the ideas of $B$-canonical Frobenius 
splittings. This theory is presented in Chapter
\ref{frobsplit} in a generality which is more 
than necessary for obtaining the described Frobenius
splitting results for $G$-Schubert varieties. However,
we hope that this theory could be useful elsewhere 
and we certainly consider it to be of independent 
interest.

This paper is organized in the following way. In Section 2
we introduce notation, and in Section 3 we briefly define 
Frobenius splitting and explain its fundamental ideas. 
Section 4 is devoted to some results on linearized sheaves
which should all be well known. In Section 5 we study the Frobenius
splitting of varieties of the form $G \times_P X$ for a variety $X$
with an action by a parabolic subgroup $P$. The main idea is to
decompose the Frobenius morphism on $G \times_P X$ into maps
associated to the Frobenius morphism on the base $\nicefrac{G}{P}$
and the fiber $X$ of the natural morphism $G \times_P X \rightarrow
\nicefrac{G}{P}$. In Section 6 we relate $B$-canonical Frobenius
splittings to the material in Section 5. Section 7 contains
applications of Section 5 to general $G \times G$-varieties. 
In section 8 we define the $G$-stable pieces and $G$-Schubert varieties.
In Section 9 we apply the material of the previous sections to the
class of equivariant embeddings and obtain Frobenius splitting 
results for $G$-Schubert varieties. Section 10 contains results
related to cohomology of line bundles on $G$-Schubert varieties.
Section 11 contains an example of a non-normal $G$-Schubert variety. 
Finally Section 12 contains generalizations and
variations of the previous sections.

We would like to thank the referee for a careful reading of 
this paper and for numerous suggestions concerning the presentation. 
\section{Notation}
\label{notation}

We will work over a fixed algebraically closed field
$k$. The characteristic of $k$ will depend on the 
application. By a variety we mean a reduced and 
separated scheme of finite type over $k$. 
In particular, we allow a variety to have
several irreducible components. 

\subsection{Group setup}
We let $G$ denote a connected linear algebraic group 
over $k$. We fix a Borel subgroup $B$ and a 
maximal torus $T \subset B$. The notation $P$ is 
used for a parabolic subgroup of $G$ containing 
$B$. The set of $T$-characters is denoted by 
$X^*(T)$ and we identify this set with the set
$X^*(B)$ of $B$-characters. 

\subsection{Reductive case}
In many cases we will specialize to the case 
where $G$ is reductive. In this case we will
also use the following notation : the set of 
roots determined by $T$ is denoted by 
$R \subseteq X^*(T)$ while the
set of positive roots determined by $(B, T)$ is denoted by $R^+$.
The simple roots are denoted by $\alpha_1,\dots,
\alpha_l$, and we let $\D = \{ 1, \dots, l \}$ denote the
associated index set. The simple reflection associated 
to the simple root $\alpha_i$ is then denoted by $s_i$.
The Weyl group $W=\nicefrac{N_G(T)}{T}$ 
is generated by the simple reflections $s_i$, for $i \in \D$. 
The length of $w \in W$ will be
denoted by $l(w)$. For $J \subset \D$, let $W_J$ denote the subgroup
of $W$ generated by the simple reflection associated with the elements 
in $J$, and 
let $W^J$ (resp. ${}^J W$) denote the set of minimal length coset representatives for ${W}/{W_J}$ (resp. $W_J
\backslash W$). The element in $W$ of maximal length will be 
denoted by $w_0$, while $w_0^J$ is used for the same kind of 
element in $W_J$.
For any $w \in W$, we let $\dot w$ denote a representative of
$w$ in $N_G(T)$.
For $J \subset \D$, let $P_J \supset B$ denote the corresponding
standard parabolic subgroup and $P^-_J \supset B^-$ denote its 
opposite parabolic.
Let $L_J=P_J \cap P^-_J$ be the common Levi subgroup of $P_J$ and
$P^-_J$ containing $T$. Let $U_J$ (resp. $U_J^-$) denote the 
unipotent radical of $P_J$ (resp. $P_j^-$). When $J = \emptyset$ we also
use the notation $U$ and $U^-$ for $U_J$ and $U_J^-$ respectively. 
  When $G$ is semisimple and simply connected we may associate 
a fundamental character $\omega_i$ to each simple root $\alpha_i$.
The sum of the fundamental characters is then denoted by $\rho$.
Then $\rho$ also equals half the sum of
the positive roots.

\section{The relative Frobenius morphism}

In this section we collect some results related to the Frobenius 
morphism and to the concept of Frobenius splitting. Compared to 
other presentations on the same subject, this presentation 
differs only in its emphasis on the set 
$ \Hom_{\co_{X'}}\big((F_X)_* \mathcal O_X, \mathcal O_{X'} \big)$
(to be defined below) and not just the set of Frobenius 
splittings. Thus, the obtained results are only 
small variations of already known results as can be 
found in e.g. \cite{BK}.

\subsection{The Frobenius morphism}

By definition a variety $X$ comes with
an associated morphism
$$p_X : X \rightarrow \Spec(k),$$
of schemes. Assume that the field $k$ has positive characteristic
$p>0$.
Then the Frobenius morphism on $\Spec(k)$ is the morphism of
schemes
$$ F_k : \Spec(k) \rightarrow \Spec(k),$$
which on the level of coordinate rings is defined by
$a \mapsto a^p$. As $k$ is assumed to be
algebraically closed the morphism $F_k$ is actually
an isomorphism and we let $F_k^{-1}$ denote the
inverse morphism. Composing $p_X$ with
$F_k^{-1}$ we obtain a new variety
$$p_X' : X \rightarrow \Spec(k),$$
with underlying scheme $X$.
In the following we suppress the morphism $p_X$ from
the notation and simply use $X$ as the notation
for the variety defined by $p_X$. The variety
defined by $p'_X$ is then denoted by $X'$.

The relative Frobenius morphism on $X$ is then the
morphism of varieties :
$$ F_X : X \rightarrow X',$$
which as a morphism of schemes is the identity map
on the level of points and where the associated map
of sheaves
$$F_X^\sharp : \mathcal O_{X'} \rightarrow (F_X)_* \mathcal O_X,$$
is the $p$-th power map. A key property of the
Frobenius morphism is the relation 
\begin{equation}
\label{p-power}
(F_X)^* \cl'
\simeq \cl^p
\end{equation}
which is satisfied for every line bundle $\cl$
on $X$ (here $\cl'$ denotes the corresponding line
bundle on $X'$).

\subsection{Frobenius splitting}
\label{Fsplit}

A variety $X$ is said to be {\it Frobenius split} if the
$\co_{X'}$-linear map of sheaves :
$$F_X^\sharp : \mathcal O_{X'} \rightarrow (F_X)_*
\mathcal O_X,$$
has a section; i.e. if there exists an element
$$ s \in \Hom_{\co_{X'}}\big((F_X)_*
\mathcal O_X, \mathcal O_{X'} \big),$$
such that the composition $s \circ F_X^\sharp$ is
the identity endomorphism of  $\mathcal O_{X'} $.
The section $s$ will be called  {\it a Frobenius splitting
of $X$}. 

\subsection{Compatibility with line bundles and closed subvarieties}

Fix a line bundle $\cl$ on $X$ and a closed 
subvariety $Y$ in $X$ with sheaf of ideals 
$\mathcal I_Y$. Let $Y'$ denote the closed
subvariety of $X'$ associated to $Y$ with 
sheaf of ideals denoted by $\mathcal I_{Y'}$.
The kernel of the natural morphism 
$$ \mathcal Hom_{\co_{X'}}\big((F_X)_* \cl,
\mathcal O_{X'} \big) \rightarrow
\mathcal Hom_{\co_{X'}}\big((F_X)_*( \cl  \otimes
\mathcal  I_Y) ,
\mathcal O_{Y'} \big),$$
induced by the inclusion $ \cl  \otimes
\mathcal  I_Y \subset \cl$ and the 
projection $\co_{X'} \rightarrow \co_{Y'}$,
will be denoted by $\mathcal End_F^\cl(X,Y)$.
The associated space of global sections 
will be denoted by $\End_F^\cl(X,Y)$. When
$Y=X$ we simply denote $\mathcal End_F^\cl(X,Y)$
(resp. $\End_F^\cl(X,Y)$) by $\mathcal End_F^\cl(X)$
(resp. $\End_F^\cl(X)$). The sheaf 
$\mathcal End_F^\cl(X,Y)$ is a subsheaf of 
$\mathcal End_F^\cl(X)$ consisting of the 
elements {\it compatible with $Y$}. Moreover,
there is a natural morphism 
$$\mathcal End_F^\cl(X,Y)_{|Y} \rightarrow 
\mathcal End_F^{\cl_{|Y}}(Y),$$
where the notation $_{|Y}$ means restriction
to $Y$. 

If $Y_1, Y_2, \dots ,Y_m$ is a collection
of closed subvarieties of $X$ then the notation
$\mathcal End_F^\cl(X, Y_1, \dots, Y_m)$
(or sometimes $\mathcal End_F^\cl(X, \{ Y_i \}_{i=1}^m)$)
will denote the intersection of the
subsheaves $\mathcal End_F^\cl(X, Y_i)$ for
$i=1,\dots,m$. The set of global sections of
the sheaf
$\mathcal End_F^\cl(X, Y_1, \dots, Y_m)$
will be denoted by
$\End_F^\cl(X, Y_1, \dots, Y_m)$.

When $\cl = \co_X$ we remove $\cl$ from 
all of the above notation. In particular, the 
vectorspace $\End_F(X)$ denotes the set of
morphisms from $(F_X)_* \co_X$ to $\co_{X'}$
and thus contains the set of Frobenius 
splittings of $X$. A Frobenius splitting 
$s$ of $X$ contained in $\End_F(X,\{ Y_i \}_i)$ 
is said to be compatible with the subvarieties 
$Y_1, \dots, Y_m$. When $s$ is compatible 
in this sense it induces a Frobenius 
splitting of each 
$Y_i$ for $i=1\dots,m$.
In this case we also say that {\it $s$ compatibly Frobenius
splits $Y_1,\dots,Y_m$}. In concrete terms, this
is equivalent to 
$$ s\big((F_X)_* \mathcal I_{Y_i} \big) \subset
\mathcal I_{Y_i'}.$$
for all $i$.

\begin{lem}
\label{compatible}
Let $Y$ and $Z$ denote closed subvarieties 
in $X$ and let $s$ denote a global section 
of $\mathcal End_F^\cl (X,Z,Y)$.
\begin{enumerate}
\item $s \in \End_F^\cl(X,Y_1)$ for every 
irreducible
component $Y_1$ of $Y$.
\item If the scheme theoretic intersection
$Z \cap Y$ is reduced then $s$ is contained
in $\End_F^\cl(X,Y \cap Z)$.
\end{enumerate}
\end{lem}
\begin{proof}
Let $Y_1$ denote an irreducible component of
$Y$ and let
$$\mathcal J = s \big( (F_X)_* (\mathcal I_{Y_1} \otimes \cl)
\big) \subset \co_{X'}.$$
Let $U$ denote the open complement (in $X'$)
of the irreducible
components of $Y'$ which are different from $Y_1'$.
Then $\mathcal I_{Y_1'}$ coincides with $\mathcal I_{Y'}$
on $U$ and consequently $\mathcal J_{|U} \subset (
\mathcal I_{Y'})_{|U}$ as $s$ is compatible with $Y$.
In particular, $\mathcal J_{|U} \subset (
\mathcal I_{Y_1'})_{|U}$. We claim that this implies 
that $\mathcal J \subset \mathcal I_{Y_1'}$ : let $V$ 
denote an open subset of $X'$ and let $f$ be a section of  
$\mathcal J$ over $V$. As $\mathcal J$ is a subsheaf of
$\co_{X'}$, we may consider $f$ as a function on $V$, and
it suffices to prove that $f$ vanishes on $Y_1' \cap V$. If 
$Y_1' \cap V$ is empty then this is clear. Otherwise, $U \cap V
\cap Y_1'$ is a dense subset of $Y_1'$ and it suffices to prove 
that $f$ vanishes on this set. But this follows from the
inclusion 
 $\mathcal J_{|U} \subset (\mathcal I_{Y_1'})_{|U}$.
As a consequence $s$
is compatible with $Y_1$.
The second claim follows as the sheaf of ideals
of the intersection $Z \cap Y$ is $\mathcal I_Y
+ \mathcal I_Z$.
\end{proof}

The condition that $Z \cap Y$ is reduced, in 
Lemma \ref{compatible}, only 
ensures that $Z \cap Y$ is a variety. When 
$\cl = \co_X$ and $s$ is a Frobenius splitting 
this is always satisfied 
\cite[Prop.1.2.1]{BK}.
 
\subsection{The evaluation map}

Let $k[X']$ denote the space of global regular 
functions on $X'$. Evaluating an element $s : 
(F_X)_* \co_X \rightarrow \co_{X'}$ of $\End_F(X)$ at
the constant global function $1$ on $X$ defines
an element in $k[X']$ which we denote by ${\rm ev}_X(s)$.
This defines a morphism 
$$ {\rm ev}_X : \End_F(X) \rightarrow k[X'],$$
with the property that ${\rm ev}_X(s) = 1$ 
if and only if  $s$ is a Frobenius splitting 
of $X$.

\subsection{Frobenius $D$-splittings}
\label{dsplit}

Consider an effective Cartier divisor $D$ on $X$,
and let $\sigma_D$ denote the associated 
global section of the associated line bundle 
$\co_X(D)$. A Frobenius splitting $s$ of 
$X$ is said to be a {\it Frobenius $D$-splitting}
if $s$ factorizes as
$$ s : (F_X)_* \co_X \xrightarrow{(F_X)_* \sigma_D}
(F_X)_* \co_X(D) \xrightarrow{s_D} \co_{X'},$$
for some element $s_D$ in $\End_F^{\co_X(D)} 
\big( X \big)$. We furthermore say that the
Frobenius $D$-splitting $s$ is compatible 
with a subvariety $Y$ if $s_D$ is compatible with 
$Y$.  The following result assures that, in
this case, the compatibility with closed subvarieties 
agrees with the usual definition  \cite[Defn.1.2]{Ram}.

\begin{lem}
\label{lemdsplit}
Assume that $s$ defines a Frobenius $D$-splitting
of $X$. Then $s_D$ is compatible with $Y$ if and only
if (i) $s$ compatibly Frobenius splits $Y$ and
(ii) the support of $D$ does not contain any irreducible
components of $Y$.
\end{lem}
\begin{proof}
The {\it if} part of the statement follows from
\cite[Prop.1.4]{Ram}. So assume that $s_D$ is compatible
with $Y$. Then $s_D$ induces a morphism
$$\overline{s}_D : (F_Y)_* \co_X(D)_{
|Y} \rightarrow
\co_{Y'},$$
satisfying $\overline{s}_D((\sigma_D)_{|Y})$ is the
constant function $1$ on $Y'$. As a consequence
$(\sigma_D)_{|Y}$ does not vanish on any of the
irreducible components of $Y$. This proves part
(ii) of the statement. Part (i) is clearly
satisfied.
\end{proof}

It follows that if $s$ is compatible with $Y$
and, moreover, defines a Frobenius $D$-splitting of $X$
then  $D \cap Y$ makes sense as an effective
Cartier divisor on $Y$ and, in this case, $s$ induces a
Frobenius $D \cap Y$-splitting of $Y$.

\subsection{Stable  Frobenius splittings along divisors}

Let $X^{(0)}= X$ and define recursively $X^{(n)} = (X^{(n-1)})'$
for $n \geq 1$. Composing the Frobenius morphisms on $X^{(i)}$
for $i=0,\dots, n$, we obtain a morphism
$$F_X^{(n)} : X \rightarrow X^{(n)},$$
with an associated map of sheaves
$$( F_X^{(n)})^\sharp : \co_{X^{(n)}} \rightarrow 
(F_X^{(n)})_* \co_X.$$
Let, as in Section \ref{dsplit}, $D$ denote an effective
Cartier divisor on $X$ with associated canonical section 
$\sigma_D$ of $\co_X(D)$. We say that  {\it $X$ admits 
a stable Frobenius splitting along $D$} if  
there exists a positive integer $n$ and an element 
 $$ s \in \Hom_{\co_{X^{(n)}}}\big(( F_X^{(n)})_*
\mathcal \co_X(D), \mathcal \co_{X^{(n)}} \big),$$
such that the composed map 
$$ \co_{X^{(n)}} \xrightarrow{( F_X^{(n)})^\sharp}  (F_X^{(n)})_* \co_X \xrightarrow{ (F_X^{(n)})_* \sigma_D} 
 (F_X^{(n)})_*  \co_X(D) \xrightarrow{s}  \co_{X^{(n)}} ,$$
is the identity map on $\co_{X^{(n)}}$. The element 
$s$ is called a {\it stable Frobenius splitting of $X$
along $D$}. When $Y$ is a closed subvariety of $X$ we say that
the stable Frobenius splitting {\it  $s$ is compatible with 
$Y$}  if
$$s \big( (F_X^{(n)})_* (\mathcal I_Y \otimes
\co_X(D)) \big) \subset \mathcal I_{Y^{(n)}}.$$
Notice that this condition necessarily implies that 
the support of $D$ does not contain any of the 
irreducible components of $Y$ (cf. proof of Lemma
\ref{lemdsplit}). 
Notice also that if $X$ admits a Frobenius $D$-splitting
which is compatible with $Y$ then $X$ admits a 
stable Frobenius splitting along $D$ which is 
compatible with $Y$.
The following is well known (see e.g. \cite[Lem.4.4]{T})

\begin{lem}
\label{sum}
Let $D_1$ and $D_2$ denote effective Cartier 
divisors on $X$ and let $Y$ denote a closed
subvariety of $X$. Then $X$ admits stable Frobenius
splittings along $D_1$ and $D_2$ which are 
compatible with $Y$ if and only if $X$ admits a 
stable Frobenius splitting along $D_1+D_2$ which 
is compatible with $Y$.
\end{lem}

The following result explains one of the
main applications of (stable) Frobenius
splitting. Remember that a line bundle 
$\cl$ is nef if $\cl
\otimes \cm$ is ample whenever $\cm$ is ample.

\begin{prop}
\label{prop stable}
Assume that $X$ admits a stable Frobenius splitting
along an effective Cartier divisor $D$. Then
there exists a positive integer $n$ such that for each line bundle
$\cl$ on $X$ we have an inclusion of abelian
groups
$$ {\rm H}^i(X , \cl) \subset  {\rm H}^i(X, \cl^{p^n} \otimes \co_X(D)).$$
In particular, if $D$ is ample and $\cl$ is
nef, then
$ {\rm H}^i(X , \cl)=0$ for $i>0$.
Moreover, if the stable Frobenius splitting of $X$ is 
compatible with a subvariety $Y$, $D$ is ample and $\cl$ is 
nef then the restriction morphism
$$ {\rm H}^0(X , \cl) \rightarrow  {\rm H}^0(Y, \cl),$$
is surjective.
\end{prop}
\begin{proof}
Argue as in the proof \cite[Prop.1.13(i)]{Ram}.
\end{proof}

\subsection{Duality for $F_X$}
\label{duality}
By duality (see \cite[Ex.III.6.10]{Har2})
for the finite morphism $F_X$ we may
to each quasi-coherent $\co_{X'}$-module $\cf$
associate an $\co_X$-module denoted by $(F_X)^! \cf$
and satisfying
$$(F_X)_* (F_X)^! \cf = \mathcal Hom_{\co_{X'}}
\big( (F_X)_* \co_X , \cf \big).$$
Actually, as $F_X$ is the identity on the level
of points we may define $(F_X)^! \cf$ as the
sheaf of abelian groups
$$\mathcal Hom_{\co_{X'}}
\big( (F_X)_* \co_{X} , \cf \big),$$
with $\co_X$-module structure defined by
$$ (g \cdot \phi)(f) = \phi(gf),$$
for $g ,f \in \co_X$ and
$\phi \in \mathcal Hom_{\co_{X'}}
\big( (F_X)_* \co_{X} , \cf \big) $. When
$\cf = \co_X$ we will also  use the notation
$\mathcal End^!_F(X)$ for $(F_X)^! \co_X$.
This sheaf is particularly nice when $X$ is 
smooth as $(F_X)^! \co_X$ then coincides with the 
line bundle $\omega_X^{1-p}$, where 
$\omega_X$ denotes the dualizing sheaf of 
$X$ (see e.g. \cite[Sect.1.3]{BK}).
If $Y_1, Y_2, \dots ,Y_m$ is a collection
of closed subvarieties of $X$ then
$\mathcal End^!_F(X, Y_1, \dots, Y_m)$
(or $\mathcal End^!_F(X, \{ Y_i \}_{i=1}^m)$)
will denote the subsheaf of
$\mathcal End^!_F(X)$ consisting of 
the elements mapping 
the sheaf of ideals
$\mathcal I_{Y_i} $
to $\mathcal I_{Y_i'}$
for all $i=1, \dots,m$. We say that 
$\mathcal End^!_F(X, \{ Y_i \}_{i=1}^m)$
is the subsheaf of elements compatible
with $Y_1,\dots,Y_m$.

More generally, duality for $F_X$ implies 
that we have a natural identification 
$$ (F_X)_*  \mathcal Hom_{\co_{X}} \big(\mathcal G  , 
(F_X)^!\mathcal F \big)
\simeq  
\mathcal   Hom_{\co_{X'}}
\big( (F_X)_* \mathcal G  , \mathcal F \big),$$
whenever $\mathcal G$ (resp. $\mathcal F$) is a 
quasicoherent sheaf on $X$ (resp. $X'$). This 
leads to the identification 
$$ \Hom_{\co_{X}} \big(\mathcal G  , 
(F_X)^!\mathcal F \big)
\simeq  
\Hom_{\co_{X'}}
\big( (F_X)_* \mathcal G  , \mathcal F \big),$$
where a morphism 
$\eta : \mathcal G  \rightarrow 
(F_X)^!\mathcal F $ is identified with the
composed morphism
$$\eta' : (F_X)_* \mathcal G \xrightarrow{(F_X)_* \eta} 
(F_X)_* (F_X)^!\mathcal F \simeq 
\mathcal Hom_{\co_{X'}}
\big( (F_X)_* \co_X , \cf \big) \rightarrow
\mathcal F.$$
Here the latter map is the natural evaluation 
map at the element $1$ in $\co_{X}$. From now
on we will specialize to the case where $\mathcal 
F = \co_{X'}$ and $\mathcal G$ equals a line bundle
$\cl$ on $X$. In this case, an element in 
$\Hom_{\co_{X}} \big( \cl  
, \mathcal End^!_F(X)\big)$
may also be considered as a global section of
the sheaf  $\mathcal End^!_F(X) \otimes \cl^{-1}$.
For later
use we emphasize 

\begin{lem}
\label{duality1}
Let $\eta$ be an element in $\Hom_{\co_{X}} \big( \cl  
, \mathcal End^!_F(X)\big)$
and let $\eta'$ denote the corresponding
element in 
$ \Hom_{\co_{X'}} \big( (F_X)_* \cl  
, \co_{X'} \big)$ by the above identification.
Then $\eta'$ factors through the morphism  
$$ (F_X)_* \mathcal \cl \xrightarrow{(F_X)_*  \eta }
 (F_X)_*\mathcal End^!_F(X).  $$
Moreover, the element $\eta'$ is compatible with 
a collection of closed subvarieties $Y_1,\dots,Y_m$ of $X$ if and 
only if the image of $\eta$ is contained in 
$ \mathcal End^!_F(X,Y_1,\dots,Y_m)$. 
\end{lem}
\begin{proof}
The first part of the statement follows directly
from the discussion above. To prove the second
statement we may assume that $m=1$. We use
the notation $Y = Y_1$. Let $\sigma$ denote 
a section of $\cl$ over an open subset $U$
of $X$, and consider $s=\eta(\sigma)$ as a
map 
$$ s :  \co_{X}(U) \rightarrow \co_{X'}(U').$$ 
That $s$ is compatible with $Y$ means that 
$s(f)$ vanishes on $Y'$ whenever $f$ vanishes on 
$Y$ for a function $f$ on $U$. Alternatively,
the evaluation of $f \cdot s$ at $1$, 
which coincides with $\eta'(f 
\cdot \sigma)$, should vanish
on $Y'$. In particular, the image of $\eta$
is contained in $ \mathcal End^!_F(X,Y)$ if
and only if the restriction of $\eta'$ to $(F_X)_* 
\big( \mathcal I_Y \otimes \cl \big)$  
maps into $\mathcal I_{Y'}$. This ends the 
proof.
\end{proof}

We will also need the following remark
\begin{lem}
\label{duality2}
Let $D$ denote a reduced effective Cartier divisor on
$X$ and $\cl$ denote a line bundle on $X$. Let
$\cm = \co_X((p-1)D) \otimes \cl$ and assume that
we have an $\co_X$-linear  morphism
$\eta : \mathcal M \rightarrow \mathcal End^!_F(X)$. Let $\sigma_D$ denote the canonical
section of $\co_X(D)$ and consider the map
$$ \eta_D : \cl \rightarrow \mathcal End^!_F(X),$$
induced by $\sigma_D^{p-1}$. Then the element
$$\eta'_D \in \mathcal Hom_{\co_{X'}}
\big( (F_X)_* \cl , \co_{X'} \big),$$
induced by $\eta_D$, is compatible with
the support of $D$. In
particular, the image of $\eta_D$ is
contained in $\mathcal End^!_F(X,D)$.
\end{lem}
\begin{proof}
Notice that $\eta_D'$ is the composition
$$ \eta_D' :(F_X)_* {\cl} \xrightarrow{(F_X)_* \sigma_D^{p-1}}
(F_X)_* {\cm} \xrightarrow{\eta'} \co_{X'},$$
where $\eta'$ is the element corresponding to $\eta$.
Hence, the restriction of $\eta_D'$ to $\cl \otimes
\co_X(-D)$ coincides with the map
$$ (F_X)_* \big({\cl \otimes \co_X(-D)}\big) \xrightarrow{(F_X)_* \sigma_D^{p}}
(F_X)_* {\cm} \xrightarrow{\eta'} \co_{X'}.$$
But the restriction of $\eta'$ to (cf. (\ref{p-power}))
$$(F_X)_*\big( \co_X(-p D) \otimes \cm \big)
\simeq  \co_{X'}(-D') \otimes   (F_X)_* \cm,$$
maps by linearity into $ \co_{X'}(-D')$. The 
{\it in particular} part follows by Lemma
\ref{duality1}.
\end{proof}

\subsection{Push-forward operation}
\label{push-forward}
Assume that $f : X \rightarrow Z$ is a morphism
of varieties satisfying that the associated map
$f^\sharp : \co_Z \rightarrow f_* \co_X$ is an
isomorphism. Let $f' : X' \rightarrow Z'$
denote the associated morphism. Then $f'_*$
induces a morphism
$$ f'_* \mathcal End_F(X) \rightarrow \mathcal
End_F(Z).$$
If $Y \subset X$ is a closed subset 
then the subsheaf  $f'_* \mathcal End_F(X,Y)$ is mapped to 
$\mathcal  End_F(Z,\overline{f(Y)} )$, where
$\overline{f(Y)}$ denotes the variety associated
to the closure of the image of $Y$. On the
level of global sections this means that every
Frobenius splitting $s$ of $X$ induces a
Frobenius splitting $f'_* s$ of $Z$ such that when
$s$ is compatible with $Y$ then $f'_* s$
is compatible with $\overline{f(Y)}$.
Likewise

\begin{lem}
\label{lem push forward}
With notation as above, let $\cl$ denote a line
bundle on $Z$ and let $s$ be an element of
$\End_F^{f^* (\cl)} \big(X  \big)$.
Then $f'_* s$ is an element of
$\End_F^{\cl} \big(Z  \big)$.
Moreover, if $s$ is compatible with a closed
subvariety $Y$ of $X$ then $f'_* s$ is compatible
with $\overline{f(Y)}$.
\end{lem}
\begin{proof}
This follows easily from the fact that the
sheaf of ideals of $\overline{f(Y)}$ coincides
with $f_* \mathcal I_Y$ \cite[Lem.1.1.8]{BK}.
\end{proof}

\section{Linearized sheaves}
\label{lin-sheaf}
In this section we collect a number of well known facts about 
linearized sheaves. The chosen presentation follows rather 
closely the presentation in \cite[Sect.2]{Bri}. 

Let $H$ denote a linear algebraic group over the field $k$ and
let $X$ denote a $H$-variety with $H$-action defined by
$\sigma : H \times X \rightarrow X$. We let
$p_2:  H \times X \rightarrow X$ denote projection on
the second coordinate. A {\it $H$-linearization} of a quasi-coherent sheaf $\cf$ on
$X$ is an $\co_{H \times X}$-linear isomorphism
$$ \phi : \sigma^* \cf \rightarrow p_2^* \cf,$$
satisfying the relation
\begin{equation}
\label{lin-rel}
(\mu \times {\bf 1}_X)^* \phi =
p_{23}^* \phi \circ ({\bf 1}_H \times \sigma)^* \phi,
\end{equation}
as morphisms of sheaves on $H \times H \times X$.
Here $\mu : H \times H \rightarrow H$ (resp. $p_{23} :
H \times H \times X \rightarrow H \times X$) denotes the
multiplication on $H$ (resp. the projection on the second
and third coordinate). Based on the fact that 
$\sigma^* \co_X  = p_2^* \co_X$ we see that the sheaf 
$\co_X$ admits a canonical linearization. In the 
following we will always assume that $\co_X$ is 
equipped with this canonical linearization.

A morphism $\psi : \cf \rightarrow \cf'$ of $H$-linearized
sheaves is a morphism of $\co_X$-modules commuting with
the linearizations $\phi$ and $\phi'$ of $\cf$ and $\cf'$,
i.e. $\phi' \circ \sigma^*(\psi) = p_2^*(\psi) \circ \phi$.

Linearized sheaves on $X$ form an abelian category which we denote by $Sh_H(X)$.

\subsection{Quotients and linearizations}

Assume that the quotient $q: X \to \nicefrac{X}{H}$ exists and that $q$ is a locally
trivial principal $H$-bundle. Then for $\cg \in Sh(\nicefrac{X}{H})$, $q^* \cg$
is naturally a $H$-linearized sheaf on $X$. This defines a functor
$q^*: Sh(\nicefrac{X}{H}) \to Sh_H(X)$. On the other hand, for $\cf \in
Sh_H(X)$, $q_* \cf$ has a natural action of $H$. Define a functor
$q_*^H: Sh_H(X) \to Sh(\nicefrac{X}{H})$ by $q_*^H(\cf)=(q_* \cf)^H$ the
subsheaf of $H$-invariants of $q_* \cf$. It is known that the
functor $q^*: Sh(\nicefrac{X}{H}) \to Sh_H(X)$ is an equivalence of categories with
inverse functor  $q_*^H$.

In general, if $H$ is a closed normal subgroup of $G$ and 
$X$ is a $G$-variety such that the quotient $\nicefrac{X}{H}$ exists (as above), 
then $\nicefrac{X}{H}$ is a $\nicefrac{G}{H}$-variety and the functor $q^*: Sh_{\nicefrac{G}{H}} (\nicefrac{X}{H}) \to Sh_G(X)$ is an equivalence of categories with inverse functor  $q_*^H: Sh_G(X) \to Sh_{\nicefrac{G}{H}}(\nicefrac{X}{H})$.

\subsection{Induction equivalence}
\label{Equivalence}
Consider now a connected linear algebraic
group $G$ and a parabolic subgroup $P$ in
$G$. Let $X$ denote a $P$-variety. Then
$G \times X$ is a $G \times P$-variety
by the action
$$(g,p) (h,x) = (ghp^{-1}, px),$$
for $g,h \in G$, $p \in P$ and $x \in X$.
Then the quotient, denoted by $G \times_P X$,
of $G \times X$ by $P$ exists and the  
associated quotient map $q : G \times X \to 
G \times_P X $ is a locally trivial principal
$P$-bundle.  
The quotient of $G \times X$ by $G$ also exists and  
may be identified with the projection $p_2 : G \times X \rightarrow X$.
In particular, we may apply the above
consideration to obtain equivalences
of the categories $Sh_P(X)$, $Sh_{G \times P} (G \times X)$ and $Sh_G(G \times_P X)$. 
Notice that
under this equivalence a $P$-linearized
sheaf $\cf$ on $X$ corresponds to the $G$-linearized
sheaf ${\mathcal Ind}^G_P(\cf)=(q_* p_2^*\cf)^P$. In
particular, the space of global sections of ${\mathcal Ind}^G_P(\cf)$
equals
\begin{align}
\label{Ind}
{\mathcal Ind}^G_P(\cf)(G \times_P X) &= \big(p_2^*\cf(G \times X)\big)^P \nonumber \\
&= \big(k[G] \otimes_k \cf(X)\big)^P 
 \\
&={\rm Ind}_P^G(\cf(X)), \nonumber 
\end{align}
where the second equality follows by the
K\"unneth formula. This also explains the
notation ${\mathcal Ind}_P^G(\cf)$.
Similarly, starting with a $G$-linearized
sheaf $\cg$ on $G \times_P X$ then the
associated
$P$-linearized line bundle on $X$ equals
$\cg' = ((p_2)_* q^* \cg)^G$. However,
by  \cite[Lemma 2(1)]{Bri} the latter
also equals the simpler pull back
$i^* \cg$ by the $P$-equivariant map
$$i : X \rightarrow G \times_P X,$$ sending $x$ to $q(1, x)$. 
In particular, we
conclude that the functor $i^*: Sh_G(G \times_P X) \to Sh_P(X)$ is an equivalence of categories with inverse functor  ${\mathcal Ind}_P^G$. Notice also that the
space of global sections of $\cg$ is 
$G$-equivariantly isomorphic to
$$ \cg(G \times_P X) = {\rm Ind}_P^G\big((i^* \cg)(X)\big),$$
which follows by (\ref{Ind}) above.

\subsection{Duality}
\label{lin-dual}
Assume that the field $k$ has positive characteristic $p>0$.
Regard $X'$ as a $H$-variety in the canonical 
way and let $\cf$ denote a $H$-linearized sheaf on $X'$. The 
sheaf $(F_X)^! \cf$, defined in Section \ref{duality}, is 
then naturally a $H$-linearized sheaf on $X$. Moreover, 
the induced $H$-linearization of $(F_X)_* (F_X)^! \cf$ 
coincides with the natural $H$-linearization of 
$$ \mathcal Hom_{\co_{X'}} \big( (F_X)_* \co_X , \cf\big).$$
When $X$ is smooth the sheaf $(F_X)^! \co_{X'}$ is 
canonically isomorphic to $\omega_X^{1-p}$ (cf. Section 
\ref{duality}). We may use this isomorphism to define 
a $H$-linearization of $\omega_X^{1-p}$. Alternatively 
we may consider the natural $H$-linearization of the 
dualizing sheaf $\Omega_{X}$ of $X$ and use this to 
define a $H$-linearization of $\omega_X^{1-p}$. It may 
be checked that the two stated ways of defining a 
$H$-linearization of $\omega_X^{1-p}$ coincide.

\section{Frobenius splitting of $G \times_P X$}
\label{frobsplit}
Let $G$ denote a connected linear algebraic group
over an algebraically closed field $k$ of 
characteristic $p>0$. Let $P$ denote a parabolic
subgroup of $G$ and let $X$ denote a $P$-variety.
In this section we want to consider Frobenius
splittings of the quotient $Z=G \times_P X$ of
$G \times X$ by $P$.  We let $\pi : Z \rightarrow
\nicefrac{G}{P}$ denote the morphism induced by
the projection of $G \times X$ on the first coordinate.
When $g \in G$ and $x \in X$ we use the notation
$[g,x]$ to denote the element in $Z$ represented by
$(g,x)$.

\subsection{Decomposing the Frobenius morphism}
\label{decomp}
The Frobenius morphism $F_Z$ admits a decomposition
$F_Z = F_b \circ F_f$  where $F_b$ (resp. $F_f$)
is related to the Frobenius morphism on the base
(resp. fiber) of $\pi$.  More precisely, define
$\hat{Z}$ and the morphisms $\hat \pi$ and $F_b$ as part of
the fiber product diagram
\begin{equation}
\label{fibrediagram}
\xymatrix{
\hat{Z} \ar[r]^{F_b} \ar[d]_{\hat \pi} & Z' \ar[d]^{\pi'} \\
\nicefrac{G}{P} \ar[r]^{ F_{\nicefrac{G}{P}}} & (\nicefrac{G}{P})'\\
}
\end{equation}
A local calculation shows that we may
identify $\hat Z$ with the quotient
$G \times_P X'$, where the $P$-action
on the Frobenius twist $X'$ of $X$ is
the natural one. With this identification
$\hat \pi : G \times_P X' \rightarrow \nicefrac{G}{P}$
is just the map $[g,x'] \mapsto gP$. It also
follows that the natural morphism (induced
by the Frobenius morphism on $X$)
$$ F_f : G \times_P X \rightarrow G \times_P X',$$
makes the following diagram commutative
$$
\xymatrix{
Z \ar[dr]|-{F_f} \ar@/^/[drr]^{F_Z} \ar@/_/[ddr]_{\pi} & & \\
& \hat{Z} \ar[r]^{F_b} \ar[d]_-{\hat \pi} & Z' \ar[d]^{\pi'} \\
& \nicefrac{G}{P} \ar[r]^{ F_{\nicefrac{G}{P}}} & (\nicefrac{G}{P})' \\
}
$$

\subsection{}
\label{intr}

Let $\cm$ denote a $P$-linearized line bundle on $X$ and let 
$\cm_Z = \mathcal Ind_P^G(\cm)$ denote the associated 
$G$-linearized line bundle on $Z$. The main 
aim of this section is to construct global sections
of the sheaf 
$$ \mathcal End_F^{\cm_Z} (Z) = \mathcal Hom_{\co_{Z'}}
\big ( (F_Z)_* \cm_Z , \co_{Z'} \big ). $$
To this end we fix a $P$-character $\lambda$ and 
let $\cl$ denote the associated line bundle on 
$\nicefrac{G}{P}$ (cf.  Section \ref{lin-sheaf}).
The pull back $\hat \pi^* \cl$ of $\cl$ to $\hat Z$
is then denoted by $\cl_{\hat Z}$. We then define 
the following sheaves 
$$ \mathcal End_F^{\cm_Z} (Z)_f := 
  \mathcal Hom_{\co_{\hat Z}} \big ( (F_f)_*
\cm_Z,  \co_{\hat Z} \big ), $$
$$ \mathcal End_F^{\cl_{\hat Z}}(Z)_b :=
\mathcal Hom_{\co_{Z'}} \big ( (F_b)_*
\cl_{\hat Z},  \co_{Z'} \big ), $$
with spaces of global sections denoted 
by $\End_F^{\cm_Z} (Z)_f$ and $\End_F^{\cl_{\hat Z}}(Z)_b$. Notice that when $\cm$ is 
substituted with the $P$-linearized twist
$\cm(-\lambda) := \cm \otimes k_{-\lambda}$  
then 
$$ \cm(-\lambda)_Z = \cm_Z \otimes \pi^*(\cl^{-1})
= \cm_Z \otimes (F_f)^* \cl_{\hat Z}^{-1},$$
and thus by the projection formula
\begin{equation}
\label{equ1}
\mathcal End_F^{\cm(-\lambda)_Z} (Z)_f = 
  \mathcal Hom_{\co_{\hat Z}} \big ( (F_f)_*
\cm_Z,  \cl_{\hat Z}  \big ). 
\end{equation}
Sections of $ \mathcal End_F^{\cm_Z} (Z)$ are
then constructed as compositions of global
sections of the sheaves 
$\mathcal End_F^{\cm(-\lambda)_Z} (Z)_f $ 
and $ \mathcal End_F^{\cl_{\hat Z}}(Z)_b$. More precisely,
if 
$$ v \in  \Hom_{\co_{\hat Z}} \big ( (F_f)_*
\cm_Z,  \cl_{\hat Z}   \big ), $$
$$ u \in \Hom_{\co_{Z'}} \big ( (F_b)_*
\cl_{\hat Z},  \co_{Z'} \big ), $$
are global sections of the latter sheaves, then
the composition $u \circ (F_b)_* v$ defines a global
section of $\mathcal End_F^{\cm_Z} (Z)$. 

\subsection{An equivariant setup} 
\label{equiv}

We now give equivariant descriptions
of 
the sheaves 
$\mathcal End_F^{\cm_Z} (Z)_f $ 
and $ \mathcal End_F^{\cl_{\hat Z}}(Z)_b$.

\subsubsection{A description of ${\End_F^{\cm_Z } (Z)_f }$}
\label{fiber}
Now ${\mathcal End}_F^{\cm_Z}(Z)_f$
is a $G$-linearized sheaf on $\hat Z=G \times_P X'$.
Let $Y$ denote a $P$-stable subvariety of
$X$ and let $Z_Y = G \times_P Y$ denote the associated
subvariety of $Z$ with sheaf of ideals $\mathcal I_{Z_Y}
\subset \co_{Z}$. Let  $\hat Z_Y$ denote the subvariety
$G \times_P Y'$ of $\hat Z$. Then there is
a natural morphism of $G$-linearized sheaves
$$ {\mathcal End}^{\cm_Z}_F(Z)_f  \rightarrow
{\mathcal Hom}_{\co_{\hat Z}} \big( (F_f)_* (\cm_Z \otimes \mathcal I_{Z_Y}) ,
\co_{\hat Z_Y} \big), $$
induced by the inclusion $ \mathcal I_{Z_Y} \subset
\co_Z$ and the projection $\co_{\hat Z} \rightarrow
\co_{\hat Z_Y}$. We let ${\mathcal End}^{\cm_Z}_F(Z,Z_Y)_f$
denote the kernel of the above map
and arrive at a left exact sequence of $G$-linearized
sheaves
\begin{equation}
\label{exact1}
0 \rightarrow {\mathcal End}^{\cm_Z}_F(Z,Z_Y)_f
\rightarrow {\mathcal End}^{\cm_Z}_F(Z)_f
\rightarrow {\mathcal Hom}_{\co_{\hat Z}} \big( (F_f)_* (\cm_Z
\otimes \mathcal I_{Z_Y}) ,
\co_{\hat Z_Y} \big).
\end{equation}
In particular, the space of 
global sections of ${\mathcal End}^{\cm_Z}_F(Z,Z_Y)_f$
is identified with the set of elements in
${\End}_{F}^{\cm_Z}(Z)_f$ which map $(F_f)_* (\cm_Z \otimes
\mathcal I_{Z_Y})$ to $\mathcal I_{\hat Z_Y}  \subset
\co_{\hat Z}$.
Using the observations in Section \ref{Equivalence}
we can give another description of the space of global
sections of ${\mathcal End}^{\cm_Z}_F(Z,Z_Y)_f$.
Let $i' : X' \rightarrow G \times_P X'$ denote the
morphism $i'(x') = [1,x']$. Then the functor $i'$ is exact on the
category of $G$-linearized sheaves. We want to apply
this fact on the left exact sequence (\ref{exact1}) above : notice
first that
$$  (i')^*{\mathcal End}_F^{\cm_Z}(Z)_f =
 {\mathcal Hom}_{\co_{X'}} \big( (i')^*(F_f)_* \cm_Z ,
\co_{X'} \big),$$
where, moreover, $(i')^* (F_f)_* \cm_{Z}
= (F_X)_*\cm$. Thus
$  (i')^*{\mathcal End}_F^{\cm_Z}(Z)_f =
{\mathcal End}^\cm_F(X)$.
Similarly,
$$  (i')^*{\mathcal Hom}_{\co_{\hat Z}} \big( (F_f)_* (\cm_Z \otimes 
\mathcal I_{Z_Y} ),
 \co_{\hat Z_Y}     \big) = {\mathcal Hom}_{\co_{X'}} (
(F_X)_* (\cm \otimes \mathcal I_Y), \co_{Y'}).$$
In particular, we see that
the $P$-linearized sheaf on $X'$ corresponding
to the $G$-linearized sheaf ${\mathcal End}^{\cm_Z}_F(Z,Z_Y)_f$
equals the kernel of the natural map
$$ {\mathcal End}^\cm_F(X)  \rightarrow
{\mathcal Hom}_{\co_{X'}} (
(F_X)_* \mathcal (\cm \otimes \mathcal I_Y), \co_{Y'}),$$
i.e. it equals $ {\mathcal End}^\cm_F(X,Y) $.
By Section \ref{Equivalence} the space of global sections
$\End_F^{\cm_Z} (Z,Z_Y)_f $
of  ${\mathcal End}^{\cm_Z}_F(Z,Z_Y)_f$
is then $G$-equivariantly isomorphic to
$$ {\rm Ind}_P^G \big( {\rm End}^\cm_F(X,Y)).$$
Applying the above conclusions to the 
sheaf $\cm(-\lambda)$ we find:

\begin{prop}
\label{propfiber}
There exists a $G$-equivariant isomorphism 
$$\End_F^{\cm(-\lambda)_Z} (Z)_f 
\simeq
{\rm Ind}_P^G \big( {\rm End}^\cm_F(X) \otimes k_\lambda)
$$
such that when $Y$ is a closed $P$-stable subvariety 
of $X$ then the subset of elements of
$\End_F^{\cm(-\lambda)_Z} (Z)_f$
which map $(F_f)_* (\cm_Z \otimes \mathcal I_{Z_Y})$ to 
$(\mathcal I_{\hat Z_Y} \otimes \cl_{\hat Z})
\subset \cl_{\hat Z}$  (cf. equation 
(\ref{equ1})) is identified with
$$ 
\End_F^{\cm(-\lambda)_Z} (Z,Z_Y)_f 
\simeq
{\rm Ind}_P^G \big( 
{\rm End}^\cm_F(X,Y) \otimes k_\lambda).$$
\end{prop}

\subsubsection{A description of ${\End}_F^{\cl_{\hat Z}}(Z)_b$}
\label{base}
As $\pi'$ in the fibre-diagram (\ref{fibrediagram}) is flat
the natural morphism $(\pi')^* (F_{\nicefrac{G}{P}})_* \cl \rightarrow
(F_b)_* \hat \pi^* \cl$  is an isomorphism (\cite[Prop.III.9.3]{Har2}).
Thus there is a natural isomorphism of $G$-linearized sheaves
$${\mathcal End}_F^{\cl_{\hat Z}}(Z)_b \simeq (\pi')^* {\mathcal Hom}_{\co_{(\nicefrac{G}{P})'}}
\big( (F_{\nicefrac{G}{P}})_*  \cl ,  \co_{(\nicefrac{G}{P})'} \big)
= (\pi')^* {\mathcal End}_F^\cl(\nicefrac{G}{P}).$$
Let $V$ denote a closed subvariety of $\nicefrac{G}{P}$. Then
${\mathcal End}_F^\cl(\nicefrac{G}{P},V)$
is the kernel of the natural map 
$$ {\mathcal End}_F^\cl(\nicefrac{G}{P})
\rightarrow
 {\mathcal Hom}_{\co_{(\nicefrac{G}{P})'}} \big( (F_{\nicefrac{G}{P}})_*
(\mathcal I_V \otimes \cl) ,  \co_{\co_{V'}} \big).$$
In particular, $(\pi')^*\big({\mathcal End}_F^\cl(\nicefrac{G}{P},
V)\big)$
maps into the kernel of the induced morphism
\begin{equation}
\label{equ2}  {\mathcal End}_F^{\cl_{\hat Z}}(Z)_b
\rightarrow
(\pi')^* {\mathcal Hom}_{\co_{(\nicefrac{G}{P})'}} \big( (F_{\nicefrac{G}{P}})_*
(\mathcal I_V \otimes \cl) ,  \co_{\co_{V'}} \big).
\end{equation}
Let $q : G \rightarrow \nicefrac{G}{P}$ denote the
quotient map. Then $\hat \pi^{-1}(V)$ identifies with the
quotient $q^{-1}(V) \times_P X'$. Moreover, as $\pi'$ is 
locally trivial it follows that
$\hat{\pi}^*(\mathcal I_V) = \mathcal I_{  q^{-1}(V) \times_P X'}$.
In particular, the sheaf
$$ (\pi')^*{\mathcal Hom}_{\co_{(\nicefrac{G}{P})'}} \big( (F_{\nicefrac{G}{P}})_*
(\mathcal I_V \otimes \cl) ,  \co_{\co_{V'}} \big),$$
is isomorphic to
$$ {\mathcal Hom}_{\co_{Z'}} \big( (F_{b})_*  (\mathcal I_{q^{-1}(V) \times_P X'}
\otimes \cl_{\hat Z}) ,  \co_{(q^{-1}(V) \times_P X)'} \big).$$
Thus we see that the kernel of (\ref{equ2}) is the 
subsheaf ${\mathcal End}_F^{\cl_{\hat Z}}(Z,{\pi^{-1}(V)} )_b$ 
of elements which map $(F_{b})_*  (\mathcal I_{q^{-1}(V) \times_P X'}
\otimes \cl_{\hat Z})$ to $\mathcal I_{(q^{-1}(V) \times_P X)'}$.
The global sections of this subsheaf is denote by 
${\End}_F^{\cl_{\hat Z}}(Z,{\pi^{-1}(V)} )_b$. 
In conclusion

\begin{prop}
\label{propbase}
The map $\pi '$ induces a $G$-equivariant morphism
$$ (\pi')^* : {\rm End}^\cl_F(\nicefrac{G}{P})
\rightarrow  {\End}_F^{\cl_{\hat Z}}(Z)_b.$$
Moreover, when $V$ is a closed subvariety of 
$\nicefrac{G}{P}$ then $(\pi')^*$ maps 
the subset ${\End}_F^\cl(\nicefrac{G}{P},V)$
into ${\End}_F^{\cl_{\hat Z}}(Z,{q^{-1}(V) \times_P X} )_b.$
\end{prop}

The following is also useful.

\begin{lem}
\label{lembase}
Let $Y$ denote a closed $P$-stable subvariety
of $X$ and fix notation as above. Then each element of
${\End}_F^{\cl_{\hat Z}}(Z)_b $ maps
$(F_b)_*(\mathcal I_{\hat Z_Y} \otimes \cl_{\hat Z}) $
to $\mathcal I_{(Z_Y)'}$.
\end{lem}
\begin{proof}
It suffices to show that the natural morphism
$$ {\mathcal Hom}_{\co_{Z'}}
\big( (F_b)_* \cl_{\hat Z} ,  \co_{Z'} \big)
\rightarrow
 {\mathcal Hom}_{\co_{Z'}}
\big( (F_b)_* ( \mathcal I_{\hat Z_Y} \otimes \cl_{\hat Z}),  \co_{(Z_Y)'} \big)$$
is zero. By linearity, this will follow if the natural
morphism
$$  \mathcal I_{(Z_Y)'} \otimes  (F_b)_*
\cl_{\hat Z} \rightarrow  (F_b)_* ( \mathcal I_{\hat Z_Y}
\otimes \cl_{\hat Z}),$$
is an isomorphism, which can be checked by a
local calculation.
\end{proof}

\subsection{Conclusions}

By Proposition \ref{propfiber} an element 
$v$ in the vectorspace ${\rm Ind}_P^G \big( {\rm End}^\cm _F(X) 
\otimes k_\lambda \bigr)$ defines an element
in $\End_F^{\cm(-\lambda)_Z} (Z)_f $. Moreover, by Proposition 
\ref{propbase}, an element $u \in  
{\rm End}_F^\cl(\nicefrac{G}{P})$ defines an element $(\pi')^*(u)$ 
in $ {\End}_F^{\cl_{\hat Z}}(Z)_b.$ Thus by 
the discussion in Section \ref{intr}
we obtain a $G$-equivariant map 
$$\Phi_{\cm , \lambda}^1  :   {\rm End}_F^\cl (\nicefrac{G}{P})  \otimes
{\rm Ind}_P^G \big( {\rm End}^\cm _F(X) \otimes k_\lambda \bigr)
\rightarrow  {\End}^{\cm_Z}_F(G \times_P X ),$$
defined by 
$$ \Phi_{\cm , \lambda}^1(u \otimes v) = (\pi')^* u \circ (F_b)_* v.$$
We can now prove

\begin{thm}
\label{thm1}
Let $X$ denote a $P$-variety and $\cm$ denote a
$P$-linearized line bundle on $X$. Let $\mathcal L$ 
denote the equivariant line bundle on $\nicefrac{G}{P}$ 
associated to the $P$-character $\lambda$. 
Then the $G$-equivariant map $\Phi_{\cm , \lambda}^1 $,
defined above, satisfies
\begin{enumerate}
\item When $Y$ is a $P$-stable closed subvariety of $X$
then the restriction of $\Phi_{\cm, \lambda}^1$ to the subspace :
$$   {\rm End}_F^\cl (\nicefrac{G}{P})
\otimes
{\rm Ind}_P^G \big( {\rm End}^\cm _F(X,Y) \otimes k_\lambda \big),$$
maps to $ {\End}^{\cm_Z }_F(G \times_P X ,G \times_P Y)$.
\item When $V$ denotes a closed subvariety of $\nicefrac{G}{P}$
then the restriction of $\Phi_{\cm, \lambda}^1$ to the subspace
$$ {\rm End}_F^\cl (\nicefrac{G}{P},V) \otimes
{\rm Ind}_P^G \big( {\rm End}^\cm_F(X) \otimes k_\lambda \big),$$
maps to $ {\End}^{\cm_Z }_F(G \times_P X,q^{-1} (V)\times _P X)$, where
$q : G \rightarrow \nicefrac{G}{P}$ denotes the quotient map.
\end{enumerate}
\end{thm}
\begin{proof}
 The first statement follows from Proposition
\ref{propfiber} and Lemma \ref{lembase}. The second
statement follows from Proposition \ref{propbase}
and Lemma \ref{lem1} below.
\end{proof}

\begin{lem}
\label{lem1}
Let $V$ denote a closed subset of $\nicefrac{G}{P}$.
Then every element of $ {\mathcal End}_F^{\cm(-\lambda)_Z }(Z)_f$
will map $(F_f)_*  ( \cm_Z \otimes \mathcal I_{\pi^{-1}(V)})$ to
$\mathcal I_{(\hat \pi)^{-1}(V)} \otimes \cl_{\hat Z}$.
\end{lem}
\begin{proof}
It suffices to prove that the natural morphism
$$  \mathcal I_{(\hat \pi)^{-1}(V)} \otimes (F_f)_* \cm_Z \rightarrow
(F_f)_* \big(\mathcal I_{\pi^{-1}(V)} \otimes \cm_Z\big),$$
is an isomorphism, which can be checked by a local
calculation.
\end{proof}

\subsection{}
\label{Phi}
Identify ${\rm Ind}_P^G \big (\cm(X) \big)$ with the 
space of global sections of $\cm_Z$ (cf. Equation (\ref{Ind})).
Then we can define a $G$-equivariant morphism 
\begin{equation} 
\label{e1}
{\End}^{\cm_Z}_F(G \times_P X ) \otimes  {\rm Ind}_P^G \big (\cm(X) \big)  
\rightarrow 
{\End}_F(G \times_P X ), 
\end{equation}
by mapping $s \otimes \sigma$, for $\sigma$ a global
section of $\cm_Z$ and $s : (F_Z)_* \cm_Z \rightarrow
\co_{Z'}$, to the element
$$ (F_Z)_* \co_{Z} \xrightarrow{(F_Z)_* \sigma} (F_Z)_* \cm_Z
\xrightarrow{s} \co_{Z'},$$
in ${\End}_F(G \times_P X )$. Combining $\Phi_{\cm , \lambda}^1$
with the morphism in (\ref{e1}) we obtain a $G$-equivariant 
map 
$$
\Phi_{\cm,\lambda} : {\rm End}_F^\cl (\nicefrac{G}{P})  \otimes
{\rm Ind}_P^G \big( {\rm End}^\cm _F(X) \otimes k_\lambda \bigr)
\otimes  {\rm Ind}_P^G \big (\cm(X) \big)  \rightarrow
{\End}_F(Z ),$$
where an element $u \otimes v \otimes \sigma$ in the
domain is mapped
to the composed map
\begin{equation}
\label{comp}
(F_Z)_* \co_{Z} \xrightarrow{(F_Z)_* \sigma}
(F_Z)_* \cm_Z \xrightarrow{(F_b)_* v} (F_b)_* \cl_{\hat Z}
\xrightarrow{(\pi')^*u} \co_{Z'}. 
\end{equation}
Notice that by Lemma \ref{duality1} the map $u 
\in  {\rm End}_F^\cl (\nicefrac{G}{P})$ factors 
as 
\begin{equation}
\label{fact}
 (F_{\nicefrac{G}{P}})_* \cl 
\xrightarrow{ (F_{\nicefrac{G}{P}})_* u^!} 
 (F_{\nicefrac{G}{P}})_*
\omega_{\nicefrac{G}{P}}^{1-p} \rightarrow 
\co_{ (\nicefrac{G}{P})'},
\end{equation}
where $u^!$ is some global section of the line 
bundle $\check{\cl} := \omega_{\nicefrac{G}{P}}^{1-p} \otimes 
\cl^{-1}$ associated to $u$ (cf. Section 
\ref{duality}), and the rightmost map 
is the evaluation map
with domain $(F_{\nicefrac{G}{P}})_*
\omega_{\nicefrac{G}{P}}^{1-p} = \End_F(\nicefrac{G}{P})$.
It follows that we may extend (\ref{comp})
into a commutative diagram
\begin{equation}
\label{diag1}
\xymatrix{
(F_Z)_* \co_{Z} \ar[r]^{(F_Z)_* \sigma} \ar[d]^{(F_b)_* \hat \pi^*(u^!)  } & 
(F_Z)_* \cm_Z \ar[r]^{(F_b)_* v} \ar[d]^{(F_b)_* \hat \pi^*(u^!) }  &
(F_b)_* \cl_{\hat Z} \ar[r]^{(\pi')^*u}  \ar[d]_{(F_b)_* \hat \pi^*(u^!)} 
& \co_{Z'} \\
(F_Z)_* \pi^*\check \cl \ar[r]^{}  & 
(F_Z)_*( \cm_Z \otimes  \pi^*\check \cl \ar[r]^{}  
\ar[r]^{})   &
(F_b)_*  (\hat \pi^* \omega_{\nicefrac{G}{P}}^{1-p})  \ar[ur]^{}  
}
\end{equation}
where all the vertical maps are induced by 
multiplication by $\hat \pi^*(u^!)$. Likewise the 
lower horizontal maps are induced from the upper 
horizontal maps by multiplication with 
  $\hat \pi^*(u^!)$. The triangle on the right is
induced from (\ref{fact}) by pull-back to
$Z'$.

\begin{thm}
\label{thm11}
Let $X$ denote a $P$-variety and $\cm$ denote a
$P$-linearized line bundle on $X$. Let $\mathcal L$ 
denote the equivariant line bundle on $\nicefrac{G}{P}$ 
associated to the $P$-character $\lambda$. 
Then the $G$-equivariant map $\Phi_{\cm,\lambda} $,
defined above, satisfies
\begin{enumerate}
\item When $Y$ is a $P$-stable closed subvariety of $X$
then the restriction of $\Phi_{\cm,\lambda}$ to the subspace :
$${\rm End}_F^\cl (\nicefrac{G}{P})  \otimes
{\rm Ind}_P^G \big( {\rm End}^\cm _F(X,Y) \otimes k_\lambda \bigr)
\otimes  {\rm Ind}_P^G \big (\cm(X) \big)$$ 
maps to $ {\End}^{}_F(G \times_P X ,G \times_P Y)$. 
\item When $V$ denotes a closed subvariety of $\nicefrac{G}{P}$
then the restriction of $\Phi_{\cm,\lambda}$ to the subspace :
$${\rm End}_F^\cl (\nicefrac{G}{P},V)  \otimes
{\rm Ind}_P^G \big( {\rm End}^\cm _F(X) \otimes k_\lambda \bigr)
\otimes  {\rm Ind}_P^G \big (\cm(X) \big)$$
maps to $ {\End}^{}_F(G \times_P X,q^{-1} (V)\times _P X)$, where
$q : G \rightarrow \nicefrac{G}{P}$ denotes the quotient map.
\end{enumerate}
Moreover, let $u \in {\rm End}_F^\cl (\nicefrac{G}{P})$, $v \in 
{\rm Ind}_P^G \big( {\rm End}^\cm _F(X) \otimes k_\lambda \bigr)$
and $\sigma \in  {\rm Ind}_P^G \big (\cm(X) \big)$. Then 
the element $\Phi_{\cm,\lambda}(u \otimes v \otimes \sigma)$ 
factorizes both as 
$$ (F_Z)_* \co_Z \xrightarrow{(F_Z)_* \sigma} (F_Z)_* \cm_Z \xrightarrow{s_1}
\co_{Z'},$$ 
and as 
$$  (F_Z)_* \co_Z \xrightarrow{(F_Z)_* (\sigma \otimes \pi^* u^!)} (F_Z)_* (\cm_Z \otimes  \pi^*\check \cl )
\xrightarrow{s_2}
\co_{Z'},$$ 
where $s_1$ and $s_2$ satisfies
\begin{enumerate}[i)]
\item If $v$ is contained in ${\rm Ind}_P^G 
\big( {\rm End}^\cm _F(X,Y) \otimes k_\lambda \bigr)$
then $s_1$ and $s_2$ are compatible with $G \times_P Y$.
\item If $u$ is contained in ${\rm End}_F^\cl (\nicefrac{G}{P},V)$
then $s_1$ is compatible with $q^{-1} (V)\times _P X$.
\end{enumerate} 
\end{thm}
\begin{proof}
Part (1) and (2) follows directly from Theorem \ref{thm1} and
the definition of $\Phi_{\cm,\lambda}$. The existence of $s_1$ and
$s_2$ follows by the diagram (\ref{diag1}). Finally the claims
about the compatibility of $s_1$ and $s_2$ follows from Theorem
\ref{thm1} and Lemma \ref{lembase}.
\end{proof}

\subsection{}
\label{fsplitconclusion}

We will now describe when an element in the image 
of $\Phi_{\cm, \lambda}$ defines a Frobenius splitting 
of $Z$. For this we consider the composed map
${\rm ev}_Z \circ \Phi_{\cm, \lambda}$. Recall that 
an element $s \in \End_F(Z)$ is a Frobenius 
splitting of $Z$ if and only if ${\rm ev}_Z(s)$
is the constant function $1$ on $Z'$. 

Let $u \in {\rm End}_F^\cl (\nicefrac{G}{P})$, 
$v \in  {\rm Ind}_P^G \big( {\rm End}^\cm _F(X) 
\otimes k_\lambda \bigr)$ and $\sigma \in
{\rm Ind}_P^G \big (\cm(X) \big) $. By Equation
(\ref{comp}) the image of $u \otimes v \otimes 
\sigma$ under ${\rm ev}_Z \circ \Phi_{\cm, \lambda}$
coincides with the global section of $\co_{Z'}$
determined by the composed map
\begin{equation}
\label{compev}
\co_{Z'} \xrightarrow{F_Z^\sharp} 
(F_Z)_* \co_{Z} \xrightarrow{(F_Z)_* \sigma}
(F_Z)_* \cm_Z \xrightarrow{(F_b)_* v} (F_b)_* \cl_{\hat Z}
\xrightarrow{(\pi')^*u} \co_{Z'}.
\end{equation}
We may divide this composition into two parts. 
The first part
\begin{equation}
\notag
\co_{Z'} \xrightarrow{F_Z^\sharp} 
(F_Z)_* \co_{Z} \xrightarrow{(F_Z)_* \sigma}
(F_Z)_* \cm_Z \xrightarrow{(F_b)_* v} (F_b)_* 
\cl_{\hat Z}
\end{equation}
is defined by $\sigma$ and $v$ and    
defines a global section of $\cl_{\hat Z}$. 
The corresponding map 
$$ \Phi_{\cm , \lambda}^2  :   {\rm Ind}_P^G \big( {\rm End}^\cm _F(X) 
\otimes k_\lambda \bigr) \otimes 
{\rm Ind}_P^G \big (\cm(X) \big) 
\rightarrow  {\rm Ind}_P^G\big(k[X'] \otimes k_\lambda\big),$$
is the map induced by the morphism
\begin{equation}
\label{map1} 
{\rm End}^\cm _F(X) 
\otimes \cm(X) \rightarrow k[X'],
\end{equation}
mapping $s : (F_X)_* \cm \rightarrow \co_{X'}$ 
and $\tau$ a global section
of $\cm$, to $ s(\tau)$. Notice
that we here identify 
${\rm Ind}_P^G \big( k[X'] \otimes k_\lambda \big) $
with the space of global sections of $\cl_{\hat Z}$
(cf. Equation (\ref{Ind})). The second part takes 
a global section $\tilde \tau$  of $\cl_{\hat Z}$
and an element $u$ in ${\rm End}_F^\cl (\nicefrac{G}{P}) $
to the global section of $\co_{Z'}$ defined by 
$$ \co_{Z'}  \xrightarrow{F_b^\sharp} 
(F_b)_* \co_{\hat Z} \xrightarrow{(F_b)_* \tilde \tau}
(F_b)_* \cl_{\hat Z}
\xrightarrow{(\pi')^*u} \co_{Z'}.$$
The corresponding  map is  
$$ \Phi_\lambda : \End_F^\cl 
(\nicefrac{G}{P})  \otimes {\rm Ind}_P^G 
\big( k[X'] \otimes k_\lambda \big)  \rightarrow 
k[Z'],$$
which maps $u \otimes \tilde \tau$, to 
$((\pi')^* u) (\tilde \tau)$ (cf. Proposition 
\ref{propbase}). The restriction of $\Phi_\lambda$ :
\begin{equation}
\label{phi1}
\phi_\lambda : \End_F^\cl (\nicefrac{G}{P})  \otimes 
{\rm Ind}_P^G \big(  k_\lambda \big) \rightarrow
k,
\end{equation}
is the map corresponding to $\Phi_\lambda$ 
in case $X$ is the one point space ${\rm Spec}(k)$ (in
which case $k[X']$ is just $k$). In combination 
this defines us a commutative diagram
\begin{equation}
\label{comm}
\xymatrix{
{\rm End}_F^\cl (\nicefrac{G}{P})  \otimes
{\rm Ind}_P^G \big( {\rm End}^\cm _F(X) \otimes k_\lambda \bigr)
\otimes  {\rm Ind}_P^G \big (\cm(X) \big)
\ar[d]^{{\rm Id} \otimes \Phi_{\cm,\lambda}^{2}}  \ar[r]^(0.8){\Phi_{\cm , \lambda}}
&  \End_F(Z)
\ar[d]_{{\rm ev}_Z} \\
 \End_F^\cl (\nicefrac{G}{P})  \otimes 
{\rm Ind}_P^G \big( k[X'] \otimes k_\lambda \big) 
\ar[r]^(0.65){\Phi_\lambda}  & k[Z'] \\
 \End_F^\cl (\nicefrac{G}{P})  \otimes 
{\rm Ind}_P^G \big(k_\lambda \big) \ar[d]_{m_\lambda} \ar@{^{(}->}[u] \ar[r]^(0.75){\phi_\lambda} 
& k \ar@{^{(}->}[u] \\
{\rm End}_F\big(\nicefrac{G}{P}\big) \ar[ur]_{{\rm ev}_{\nicefrac{G}{P}}} & 
}
\end{equation}
where $m_\lambda$ is the natural map which
makes the lower part of the diagram commutative. 
Notice that when $k[X']=k$, e.g. if $X'$ is a complete 
and irreducible variety, then $\phi_\lambda$ 
and $\Phi_\lambda$ coincides. Let $\chi$ denote the
$P$-character associated to the canonical 
$G$-linearization of $\omega_{\nicefrac{G}{P}}^{-1}$
(cf. Section \ref{lin-dual}). Then as noted 
earlier (Section \ref{Phi}) the $G$-module
$\End_F^\cl (\nicefrac{G}{P})$ coincides 
with the space of global sections of 
$\check{\cl} = \omega_{\nicefrac{G}{P}}^{1-p}
\otimes \cl^{-1}$ and thus coincides with 
\begin{equation}
\label{chi}
 \End_F^\cl (\nicefrac{G}{P})
= {\rm Ind}_P^G \big( (p-1)\chi - \lambda\big),
\end{equation}
where we abuse notation and write $(p-1) \c - \lambda$ 
for the 1-dimensional  $P$-representation associated 
with the character  $(p-1) \c - \lambda$.
It follows that $m_\lambda$ is the
natural multiplication map
\begin{equation} 
\label{mult}
m_\lambda : {\rm Ind}_P^G \big(  (p-1)\chi - \lambda\big)
\otimes {\rm Ind}_P^G \big(\lambda\big)
\rightarrow  {\rm Ind}_P^G \big(  (p-1)\chi \big).
\end{equation}
which is surjective if the domain is nonzero,
i.e. if $\cl$ and $\omega_{\nicefrac{G}{P}}^{1-p} \otimes 
\cl^{-1}$ are effective line bundles on $\nicefrac{G}{P}$
\cite[Thm.3]{RamRam}. 

The commutativity of the
diagram (\ref{comm}) then implies:

\begin{prop}
\label{testFsplit}
Let $\Xi$ denote an element in the domain of 
$\Phi_{\cm,\lambda}$, and assume that the image $
({\rm Id} \otimes \Phi_{\cm,\lambda}^{2})(\Xi)$ is 
contained in the subspace $\End_F^\cl (\nicefrac{G}{P})  
\otimes {\rm Ind}_P^G \big(  k_\lambda \big)$ (cf.
diagram (\ref{comm})). Then 
$\Phi_{\cm,\lambda}(\Xi)$ is a Frobenius splitting of 
$Z$ if and only if $\phi_\lambda(({\rm Id} \otimes 
\Phi_{\cm,\lambda}^{2})(\Xi))$ equals the constant
$1$. In particular, if 
 $\End_F^\cl (\nicefrac{G}{P})  
\otimes {\rm Ind}_P^G \big(  k_\lambda \big)$
is nonzero and ${\rm Ind}_P^G 
\big(  k_\lambda \big)$ is contained in 
the image of $\Phi_{\cm,\lambda}^{2}$, 
then $Z$ admits a Frobenius splitting. 
\end{prop}
\begin{proof}
The first part of the proof is just a 
restatement of the fact that the diagram
(\ref{comm}) is 
commutative. The second part follows 
by the surjectivity of $m_\lambda$ and 
the fact that $\nicefrac{G}{P}$ admits 
a Frobenius splitting.

\end{proof}

\begin{cor}
\label{complete}
Assume that $X$ is irreducible and complete.
If both ${\rm Ind}_P^G \big( \lambda \big)$ 
and  ${\rm Ind}_P^G \big( (p-1)\chi - \lambda \big)$ 
are nonzero and $\Phi_{\cm,\lambda}^2$ is surjective, 
then $Z$ admits a Frobenius splitting. 
\end{cor}

\subsection{}
\label{invar}
In many concrete situation the existence of a 
$P$-invariant element in ${\rm End}^\cm _F(X) 
\otimes k_\lambda$ is given. Notice that this
is equivalent to a $G$-invariant element $v$ 
in ${\rm Ind}_P^G \big( {\rm End}^\cm _F(X) 
\otimes k_\lambda \bigr)$ and thus 
$\Phi_{\cm,\lambda}$ defines a $G$-equivariant map 
\begin{equation}
\label{invarsect}
{\rm End}_F^\cl (\nicefrac{G}{P})  \otimes
{\rm Ind}_P^G \big (\cm(X) \big)  \rightarrow
{\End}_F(Z ),
\end{equation}
$$ u \otimes \sigma \mapsto 
\Phi_{\cm,\lambda} (u \otimes v \otimes \sigma).$$
Similarly $\Phi_{\cm,\lambda}^2$ defines a $G$-equivariant
morphism 
\begin{equation}
\label{invarsect1}
{\rm Ind}_P^G \big (\cm(X) \big) 
\rightarrow {\rm Ind}_P^G \big ( k[X'] \otimes 
k_\lambda \big),
\end{equation}
which makes the diagram 
\begin{equation}
\label{comm3}
\xymatrix{
{\rm End}_F^\cl (\nicefrac{G}{P})  
\otimes  {\rm Ind}_P^G \big (\cm(X) \big)
\ar[d]^{}  \ar[r]{}
&  \End_F(Z)
\ar[d]_{{\rm ev}_Z} \\
 \End_F^\cl (\nicefrac{G}{P})  \otimes 
{\rm Ind}_P^G \big( k[X'] \otimes k_\lambda \big) 
\ar[r]^(0.75){\Phi_\lambda}  & k[Z']
}
\end{equation}
commutative. We also note

\begin{cor}
\label{complete2}
Assume that $X$ is irreducible and complete
and let $v$ denote a $P$-invariant element of 
${\rm End}^\cm _F(X) \otimes k_\lambda$.
If the induced map 
$$ (\Phi_{\cm,\lambda}^2)_{| v \otimes   
{\rm Ind}_P^G  (\cm(X) )}
: {\rm Ind}_P^G \big( \cm(X) \big) 
\rightarrow  {\rm Ind}_P^G \big(k_\lambda \big), $$
is surjective then $Z$ admits a Frobenius 
splitting. In particular, if $
 {\rm Ind}_P^G \big(k_\lambda \big)$ is an 
irreducible $G$-representation then for
$Z$ to be Frobenius split it 
suffices that the latter map is nonzero.
\end{cor}
\begin{proof}
Apply Corollary \ref{complete}.
\end{proof}

\section{$B$-Canonical Frobenius splittings}
\label{canonical}

In this section we continue the study of the Frobenius
splitting properties of $Z=G \times_P X$. The notation 
is kept as in Section \ref{frobsplit} but we restrict
ourselves to the case where  $G$ is a connected, 
semisimple and simply connected linear algebraic group. 
Moreover, we fix $P=B$, $\cm  = \co_X$ and $\lambda = 
-(p-1)\rho$.  
Recall that, in this setup, the dualizing 
sheaf $\omega_{\nicefrac{G}{B}}$ is the 
$G$-linearized sheaf associated to the
$B$-character $2\rho$. Thus, 
with the notation in Section \ref{fsplitconclusion},
we have $\chi = - 2 \rho$.  Recall also the
$G$-equivariant identity (see (\ref{chi})) 
\begin{equation}
\label{iso}
{\rm End}_F^\cl (\nicefrac{G}{B})   \simeq 
{\rm Ind}_B^G ((p-1) \c - \lambda) =
{\rm Ind}_B^G (\lambda) = 
{\rm Ind}_B^G ((1-p)\rho).
\end{equation}
The latter $G$-module is called
the Steinberg module of $G$ and will be denoted
by ${\rm St}$. The Steinberg module  
is a simple and selfdual $G$-module.    
A {\it $B$-canonical Frobenius 
splitting} of $X$ is then a $B$-equivariant map
\begin{equation} 
\label{Bcan}
\theta : {\rm St} \otimes k_{(p-1)\rho} \rightarrow
{\rm End}_F(X),
\end{equation}
containing a Frobenius splitting in its image. 
Notice that a $B$-canonical Frobenius splitting
of $X$ is not a Frobenius splitting as defined 
in Section \ref{Fsplit}. However, there exists
a unique  nonzero lowest weight vector $v_-$ of 
${\rm St}$ such that $\theta(v_- )$ is a Frobenius
splitting in 
the sense of Section \ref{Fsplit}.
Moreover, as ${\rm St}$ is a simple $G$-module 
the map $\theta$ is uniquely determined 
by $\theta(v_- )$, and we may thus identify 
$\theta$ with $\theta(v_-)$.  In this 
way $\theta(v_-)$ will also be called a 
$B$-canonical Frobenius splitting of
$X$.

The importance of $B$-canonical Frobenius 
splittings was first observed by O. Mathieu 
in connection with good filtrations of $G$-modules. 
We refer to \cite[Chapter 4]{BK} for a general
reference on $B$-canonical Frobenius splittings.

\subsection{}
\label{61}
Consider a $B$-canonical Frobenius
splitting as in (\ref{Bcan}). By Frobenius 
reciprocity this defines a map
$$ {\rm St} \rightarrow
{\rm Ind}_B^G \big({\rm End}_F(X) \otimes k_\lambda
\big),$$
and as ${\rm Ind}_B^G\big(k[X] \big) $
contains $k$ we may consider the induced $G$-equivariant 
morphism
$$ \tilde \theta : {\rm St} \rightarrow 
{\rm Ind}_B^G \big({\rm End}_F(X) \otimes k_\lambda
\big) \otimes {\rm Ind}_B^G\big( k[X] \big).$$
Composing $\tilde \theta$ with the map 
$\Phi_{\cm, \lambda}^2$ of Section \ref{fsplitconclusion}
we end up with a map 
$$ \Phi_{\cm, \lambda}^2 \circ \tilde \theta :
{\rm St} \rightarrow {\rm Ind}_B^G \big(
k[X'] \otimes k_\lambda\big).$$
We claim
\begin{lem}
\label{image}
The composed map
$\Phi_{\cm, \lambda}^2 \circ \tilde \theta$
is an isomorphism on its image 
${\rm Ind}_B^G \big(\lambda\big)$.
\end{lem}
\begin{proof}
We first prove that the image of $\Phi_{\cm, \lambda}^2 
\circ \tilde \theta$ is contained in ${\rm Ind}_B^G 
\big(\lambda\big)$. For this let $\End_F(X)_c$ denote the inverse image
of $k \subset k[X']$ under the evaluation map ${\rm ev}_X$.
It suffices to prove that 
the image of $\theta$ is contained in 
$\End_F(X)_c$. Notice that $\End_F(X)_c$ is a 
$B$-submodule of $\End_F(X)$ containing the 
set of Frobenius splittings of $X$.
In particular, the image 
of the lowest weight space of ${\rm St}$ 
under $\theta$ is contained in $\End_F(X)_c$. 
Moreover, as ${\rm St}$ is an irreducible 
$G$-module it is generated by the lowest 
weight space as a $B$-module. Thus,
the image of $\theta$ will be contained
in the $B$-module $\End_F(X)_c$.

Now $\Phi_{\cm, \lambda}^2 
\circ \tilde \theta$ is a map from ${\rm St}$ to 
${\rm Ind}_B^G( \lambda) = {\rm St}$. Thus, by 
Frobenius reciprocity, it suffices 
to prove that $\Phi_{\cm, \lambda}^2 \circ \tilde 
\theta$ is nonzero which is the case as $\theta$
contains a Frobenius splitting in its image.
\end{proof}

Using Lemma \ref{image} we can now combine 
the diagram (\ref{comm}) 
with the map $\Phi_{\cm, \lambda}^2 \circ \tilde 
\theta$ and obtain a commutative and $G$-equivariant
diagram
\begin{equation}
\label{comm4}
\xymatrix{
{\rm St}
\otimes  {\rm St}
\ar[d]_{\simeq}^{ {\rm Id} \otimes (\Phi_{\cm, \lambda}^2 
\circ \tilde \theta)}  \ar[r]^(0.4){\Theta} 
&  \End_F(G \times_B X)
\ar[r]^(0.65){{\rm ev}_Z}  & k[Z'] \\
 \End_F^\cl (\nicefrac{G}{B})  \otimes 
{\rm Ind}_B^G \big(k_\lambda \big) \ar[dr]^{m_\lambda}
\ar[rr]^(0.65){\phi_\lambda}  & &k \ar[u] \\
&{\rm End}_F\big(\nicefrac{G}{B}\big) \ar[ur]_{{\rm ev}_{\nicefrac{G}{B}}} && \\
}
\end{equation}
where $\Theta$ is the map induced by 
$\tilde \theta
$ and $\Phi_{\cm,\lambda}$. By Proposition \ref{testFsplit} 
it follows that $\Theta(\Xi)$, for $\Xi$ in ${\rm St}
\otimes  {\rm St}$, is a Frobenius splitting of $Z$
if and only if the image of $\Xi$ under $ \phi_\lambda $ and 
${\rm Id} \otimes (\Phi_{\cm, \lambda}^2 \circ \tilde \theta)$
equals $1$. The latter map from ${\rm St}
\otimes  {\rm St}$ to $k$ will be denoted by 
$\phi$. By construction $\phi$ is $G$-equivariant.
Moreover, $m_\lambda$ is surjective and 
${\rm ev}_{\nicefrac{G}{B}}$ is nonzero (as $\nicefrac{G}{B}$
admits a Frobenius splitting) and thus $\phi$ is
nonzero. As ${\rm St}$ is a simple $G$-module it 
follows that 
\begin{equation}
\label{form}
\phi : {\rm St} \otimes  {\rm St} \rightarrow k,
\end{equation}
defines a nondegenerate $G$-invariant
bilinear form on ${\rm St}$. By Frobenius reciprocity 
such a form is uniquely determined up to a nonzero
constant. In particular, this provides a very useful 
way to construct lots of Frobenius splittings of $Z$.

\begin{cor}
\label{B-canonical}
Let $\theta : {\rm St} \otimes
k_{(p-1)\rho} \rightarrow {\rm End}_F(X)$
denote a $B$-canonical Frobenius
splitting of $X$. Then the induced 
morphism (defined above)
$$\Theta : {\rm St} \otimes
 {\rm St} \rightarrow {\End}_F(G \times_B X),$$
satisfies the following 
\begin{enumerate}
\item The image $\Theta(\nu)$ of an element 
$\nu $ in $ {\rm St} \otimes  {\rm St}$ defines
a Frobenius splitting of $G \times_B X$ up to a nonzero
constant if and only if $\phi(\nu)$ is nonzero.
\item If the image of $\theta $ is
contained in ${\rm End}_F(X,Y)$ for a
$B$-stable closed subvariety $Y$ of $X$,
then the image of $\Theta$ is contained
in ${\End}_F(G \times_B X, G \times_B Y)$.
\item Let $v$ denote an element of
${\rm St}= \End_F^\cl(\nicefrac{G}{B})$
which is compatible with a closed 
subvariety $V$ of $\nicefrac{G}{B}$. 
For any element $v' \in {\rm St}$ 
we have
$$\Theta(v \otimes v')
\in {\End_F}(G \times_B X, q^{-1}(V) \times_B X),$$
with $q : G \rightarrow \nicefrac{G}{B}$
denoting the quotient map.
\item Any element of the form $\Theta(v \otimes v')$
factorizes as 
$$ (F_{Z} )_* \co_Z \xrightarrow{(F_Z)_* \pi^* v} (F_Z)_* \pi^* \cl
\xrightarrow{s} \co_{Z'}, $$
where $Z= G \times_B X$ and $\cl$ is the line bundle
on $\nicefrac{G}{B}$ associated to the $B$-character 
$(1-p)\rho$. Moreover, if the image of $\theta $ is
contained in ${\rm End}_F(X,Y)$ then $s$ is compatible
with $G \times_B Y$.
\end{enumerate}
\end{cor}
\begin{proof}
All statements follows directly from
Theorem \ref{thm11} and the considerations 
above.
\end{proof}

The first part (1) and (2) of the above result
is well known (see e.g.
\cite[Ex. 4.1.E(4)]{BK}). However,
the second part (3) and (4) 
seems to be new.

\subsection{$B$-canonical Frobenius splitting 
when $G$ is not semisimple}
\label{not ss} 
Although Corollary \ref{B-canonical} is only
stated for connected, semisimple
and simply connected groups it
also applies in other cases :
assume  that $G$ is a connected linear algebraic
group containing a connected semisimple
subgroup $H$ such that the induced map
$ \nicefrac{H}{H \cap B} \rightarrow
\nicefrac{G}{B}$ is an isomorphism.
E.g. this is satisfied for any
parabolic subgroup of a reductive connected
linear algebraic group. Let $q_{\rm sc}: H_{\rm sc}
\rightarrow H$ denote a simply connected
cover of $H$. Then $X$ admits an action
of the parabolic subgroup
$B_{\rm sc}:=q_{\rm sc}^{-1}(B \cap H)$ of $H_{\rm sc}$.
Furthermore, the natural morphism
$$ H_{\rm sc} \times_{B_{\rm sc}} X
\rightarrow G \times_B X, $$
is then an isomorphism. We then say that
$X$ admits a $B$-canonical Frobenius
splitting if $X$, as a $B_{\rm sc}$-variety,
admits a $B_{\rm sc}$-canonical Frobenius
splitting. In this case we may apply
 Corollary \ref{B-canonical} to obtain
Frobenius splitting properties of
$G \times_B X$.

\subsection{Restriction to Levi subgroups}
\label{sect levi}
Return to the situation where $G$ is connected, 
semisimple and simply connected. Let $J$ be a
subset of the set of simple roots $\Delta$ and
let $G_J$ denote the commutator subgroup of 
$L_J$. Then $G_J$ is a connected, semisimple 
and simply connected linear algebraic group 
with Borel subgroup $B_J = {G_J} \cap B$ and maximal 
torus $T_J = T \cap G_J$.
We let ${\rm St}_J$ denote the associated
Steinberg module. Notice that ${\rm St}_J =
{\rm Ind}_{B_J}^{G_J}({(1-p)\rho_J})$ where
$\rho_J$ denotes the restriction of $\rho$
to $B_J$. The following should be well
known but we do not know a good reference.

\begin{lem}
\label{can restriction}
There exists a $G_J$-equivariant morphism
$$ {\rm St}_J \rightarrow {\rm St},$$
such that the $B^-_J$-invariant line of
$ {\rm St}_J $ maps surjectively to the
$B^-$-invariant line of ${\rm St}$.
In particular, if $X$ is a $G$-variety
admitting a $B$-canonical Frobenius
splitting then $X$ admits a $B_J$-canonical
Frobenius splitting as a $G_J$-variety.
\end{lem}
\begin{proof}
Let $M$ denote the $T$-stable complement to
the $B$-stable line in ${\rm St}$. Then $M$
is $B^-$-invariant and thus also
$B_J^-$-invariant. The translate $\dot w_0^J M$ 
is then invariant under $B_J$ and we obtain a
$B_J$-equivariant morphism
$$ {\rm St} \rightarrow {\rm St}/(\dot w_0^J M)
\simeq k_{(1-p)\rho_J}.$$
By Frobenius reciprocity this defines
a $G_J$-equivariant map ${\rm St}
\rightarrow {\rm St}_J$ such that the
$B$-stable line of ${\rm St}$ maps
onto the $B_J$-stable line of ${\rm St}_J$.
Now apply the selfduality of ${\rm St}_J$
and ${\rm St}$ to obtain the desired map.
This proves the first part of the statement.

The second part follows easily by composing
the obtained morphism $ {\rm St}_J \rightarrow {\rm St}$
with the $B$-canonical Frobenius splitting
$${\rm St}
\rightarrow {\rm End}_F(X) \otimes
k_{(1-p)\rho}, $$
of $X$ and noticing that the restriction
of $\rho$ to $B_J$ is $\rho_J$.

\end{proof}

\section{Applications to $G \times G$-varieties}

In this section we consider a linear algebraic group
$G$ satisfying the conditions of Section 
\ref{not ss}, i.e. we assume that $G$ contains
a closed connected semisimple subgroup $H$ such that 
$ \nicefrac{H} {H \cap B} \rightarrow \nicefrac{G}{B}$
is an isomorphism. We also let $H_{\rm sc}$ denote 
the simply connected version of $H$ and let
$B_{\rm sc}$ denote the associated Borel subgroup.

\subsection{A well known result}
\label{well known}

Consider for a moment (i.e. in this subsection) the
case where $G$ is semisimple and simply connected. 
Remember that the $G$-linearized 
line bundle on $\nicefrac{G}{B}$
associated to the $B$-character $2\rho$
coincides 
with the dualizing sheaf $\omega_{\nicefrac{G}{B}}$.
 Let $\cl$ denote
the line bundle on $\nicefrac{G}{B}$ associated to 
the $B$-character $(1-p)\rho$ and recall from 
Section \ref{canonical} the notation ${\rm St}=
{\rm Ind}_B^G ( (1-p)\rho)$ for the Steinberg module. 
As the Steinberg module is a selfdual $G$-module 
we may fix a $G$-invariant nonzero 
element $v_\Delta$ in the tensorproduct 
${\rm St} \otimes {\rm St}$. We may think of 
$v_\Delta$ as a global section of the line bundle
$\cl \boxtimes \cl$ on $(\nicefrac{G}{B})^2 =\nicefrac{G}{B}
\times \nicefrac{G}{B}$.

Identify $\nicefrac{G}{B}\times \nicefrac{G}{B}$ 
with $G \times_B \nicefrac{G}{B}$ by the isomorphism
$$ G \times_B \nicefrac{G}{B} \rightarrow 
\nicefrac{G}{B}\times \nicefrac{G}{B},$$
$$ [g, hB] \mapsto (gB, ghB), $$
and let $D$ denote the subvariety
of $\nicefrac{G}{B} \times \nicefrac{G}{B}$ 
corresponding to $G \times_B \partial (\nicefrac{G}{B})$, 
where $\partial (\nicefrac{G}{B})$ denotes the union of the 
codimension 1 Schubert varieties in $\nicefrac{G}{B}$.
Then, by \cite[proof of Thm.2.3.8]{BK}, the zero scheme 
of $v_\Delta$ equals $(p-1)D$. 
Consider then the natural morphism :
$$ \eta : (\cl \boxtimes \cl) \otimes (\cl \boxtimes \cl)
\rightarrow \omega_{(\nicefrac{G}{B})^2}^{1-p} =
\mathcal End_F^!((\nicefrac{G}{B})^2)$$
and define  
$$ \eta_{D} : (\cl \boxtimes \cl) \rightarrow 
\mathcal End_F^!((\nicefrac{G}{B})^2),$$
as in Lemma \ref{duality2}, using the identification
$ \cl \boxtimes \cl = \co_{(\nicefrac{G}{B})^2}\big( (p-1)D \big)$.
Then by Lemma \ref{duality2} the image of $\eta_{D}$
is contained in $\mathcal End_F^!((\nicefrac{G}{B})^2,D)$ and 
thus the associated element 
$$ \eta_{D}' \in \Hom_{\co_{((\nicefrac{G}{B})^2)'}} \big(  
(F_{(\nicefrac{G}{B})^2})_* (\cl \boxtimes \cl), 
\co_{((\nicefrac{G}{B})^2)'} \big),$$ 
is compatible with $D$. It follows

\begin{lem}
\label{diag}
The element in 
$$ \End_F^{\cl \boxtimes \cl} ((\nicefrac{G}{B})^2) 
  \simeq {\rm St} \boxtimes  {\rm St} $$
defined by $v_\Delta$ is compatible with 
the diagonal ${\rm diag}(\nicefrac{G}{B})$ in
$\nicefrac{G}{B} \times
\nicefrac{G}{B}$.  
\end{lem}
\begin{proof}
We have to prove that $\eta_D'$, defined above, is 
compatible with the diagonal  ${\rm diag}(\nicefrac{G}{B})$.
As $\eta_D'$ is compatible with $D$ it suffices to show 
that $\End_F^{\cl \boxtimes \cl} ((\nicefrac{G}{B})^2,D)$ is contained 
in $\End_F^{\cl \boxtimes \cl} ((\nicefrac{G}{B})^2,{\rm diag}(
\nicefrac{G}{B}) \big)$. This follows by an 
application of Lemma \ref{compatible} and an argument 
as at the end of the proof of \cite[Thm.2.3.1]{BK}. 
\end{proof}

\subsection{}

We return to the setup as in the beginning of this
section. We want to apply the results of the preceding 
sections to the case when the group equals $G \times G$. So let
$X$ denote a $B \times B$-variety and assume that $X$
admits a $B_{\rm sc} \times B_{\rm sc}$-canonical Frobenius
splitting defined by
$$ \theta : ({\rm St} \boxtimes {\rm St}) \otimes
(k_{(p-1)\rho} \boxtimes k_{(p-1)\rho}) \rightarrow
\End_F(X),$$
which is compatible with certain $B \times B$-stable
subvarieties $X_1, \dots, X_m$, i.e. the image of
$\theta$ is contained in $\End_F(X,X_i)$ for all
$i$. Then

\begin{thm}
\label{diag canonical}
The variety
$(G \times G) \times_{(B \times B)} X$
admits a $\diag(B_{\rm sc})$-canonical 
Frobenius splitting which
compatibly Frobenius splits the subvarieties
${\rm{ diag}(G)} \times_{{\rm diag}(B)} X$
and $(G \times G) \times_{(B \times B)} X_i$
for all $i$.  
\end{thm}
\begin{proof}
It suffices to consider the case where
$G = H_{\rm sc}$ (cf. discussion in Section 
\ref{not ss}).
By Corollary \ref{B-canonical}
there exists a $G \times G$-equivariant
morphism
$$ \Theta : ({\rm St} \boxtimes {\rm St})
\otimes ({\rm St} \boxtimes {\rm St})
\rightarrow \End_F((G \times G)
\times_{(B \times B)} X),$$
satisfying certain compatibility
conditions.
Let $v_\Delta \in  {\rm St}
\boxtimes {\rm St}$ be a nonzero 
${\rm diag}(G)$-invariant element and 
let $v \in  {\rm St}
\boxtimes {\rm St} $ be arbitrary.
Then by Corollary \ref{B-canonical} and 
Lemma \ref{diag} 
the element $\Theta \big(v_\Delta \otimes v 
\big)$ is compatible with ${\rm{ diag}(G)} 
\times_{{\rm diag}(B)} X$ and $(G \times G) 
\times_{(B \times B)} X_i$ for all $i$.
In particular, if we define the ${\diag(G)}$-equivariant
morphism  
 $$\Theta_{\Delta} : {\rm St} \otimes {\rm St}
\rightarrow \End_F((G \times G)
\times_{(B \times B)} X),$$
by $\Theta_{\Delta}(v) = \Theta \big(v_\Delta \otimes v 
\big)$, then every element in the image 
of $\Theta_\D$ 
is compatible with  ${\rm{ diag}(G)} 
\times_{{\rm diag}(B)} X$ and $(G \times G) 
\times_{(B \times B)} X_i$ for all $i$. 
Consider $k_{(p-1)\rho}$ as the highest weight 
line in ${\rm St}$. Then the restriction 
of $\Theta_\D$ to ${\rm St} \otimes k_{(p-1)\rho}$ 
defines a $\diag(B)$-canonical Frobenius splitting of 
$(G \times G) \times_{(B \times B)} X$ with the
desired properties.
\end{proof}

Notice that by the general machinery of 
canonical Frobenius splittings (see e.g. 
\cite[Prop.4.1.17]{BK}) the existence 
of a Frobenius splitting of 
${\rm{ diag}(G)} \times_{{\rm diag}(B)} X$
follows if $X$ admits a ${\rm diag}(B_{\rm sc})$-canonical 
Frobenius splitting. In the 
above setup $X$
only admits a $B_{\rm sc} \times B_{\rm sc}$-canonical
Frobenius splitting which is less
restrictive. However, in contrast to 
the situation when $X$ admits a ${\rm diag}(B_{\rm sc})$-canonical
Frobenius splitting, the present Frobenius 
splitting is not necessarily compatible 
with subvarieties
of the form $\overline{B \dot{w} B} \times_B X$,
with $w$ denoting an element of the Weyl group
and $\overline{B \dot{w} B}$ denoting the
closure of $B \dot{w} B$ in $G$.

\section{$G$-Schubert varieties in equivariant Embeddings}

From now on, unless otherwise stated, we assume that 
$G$ is a connected reductive group.

\subsection{Equivariant embeddings}
Consider $G$ as a $G \times G$-variety by left
and right translation. An equivariant embedding
$X$ of $G$ is then a normal irreducible $G \times G$-variety
containing an open dense subset which is
$G \times G$-equivariantly isomorphic to $G$.
In particular, we may identify $G$ with an
open subset of $X$, and the complement $X \setminus 
G$ is then called the boundary. As $G$ is an 
affine variety the boundary is of pure codimension
1 in $X$ \cite[Prop.3.1]{Har}.
Any equivariant embedding
of $G$ is a spherical variety (with respect
to the induced $B \times B$-action) and
thus $X$ contains finitely may $B \times B
$-orbits.

\subsection{Wonderful compactifications}
\label{wonderful}
When $G=G_{\rm ad}$ is of adjoint type there exists a
distinguished equivariant embedding ${\bf X}$ of
$G$ which is called the {\it wonderful compactification}
(see e.g. \cite[6.1]{BK}).

The
boundary ${\bf X} \setminus G$ is a union of irreducible divisors
${\bf X}_j$, $j \in \D$, which intersect transversely. For a
subset $J \subset \D$, we denote the intersection $\cap_{j \in J}
{\bf X}_j$ by ${\bf X}_J$. As a $(G \times G)$-variety, ${\bf X}_J$
is isomorphic to the variety $(G \times G) \times_{P_{\D \setminus
J}^- \times P_{\D \setminus J}} {\bf Y}$, where ${\bf Y}$ denotes the wonderful
compactification of the group of adjoint type associated to $L_{\D \setminus J}$. Here
the $P_{\D \setminus J}^- \times P_{\D \setminus J}$-action on ${\bf Y}
$
is defined by the quotient maps  $P_{\D \setminus J} \rightarrow
L_{\D \setminus J}$ and $P_{\D \setminus J}^- \rightarrow L_{\D
\setminus J}$. In particular, ${\bf X}_\Delta$ is $G \times G$-equivariantly
isomorphic to the variety $\nicefrac{G}{B^-} \times \nicefrac{G}{B}$.

\subsection{Toroidal embeddings}
Let $G_{\rm ad}$ denote the group of adjoint type associated
to $G$, and let $\bf X$ denote the wonderful compactification
of $G_{\rm ad}$. An embedding $X$ of the reductive group  $G$ 
is then called {\it toroidal} if the canonical map 
$G \rightarrow G_{\rm ad}$ admits an extension 
$X \rightarrow  {\bf X}$.

\subsection{$G$-Schubert varieties}
By a {\it $G$-Schubert variety} in  an equivariant 
embedding $X$ we will mean a subvariety of the form
$\diag(G) \cdot V$, for some $B \times B$-orbit closure 
$V$. Notice that $\diag(G) \cdot V$ is the image of 
$\diag(G) \times_{\diag(B)} V$ under the proper map 
$$\diag(G) \times_{\diag(B)} X \to X,$$
$$ [g,x] \mapsto g \cdot x,$$
and thus $G$-Schubert varieties are closed 
$\diag(G)$-stable subvarieties of $X$. 

If $G = G_{\rm ad}$ and $X = {\bf X}$ is the wonderful 
compactification then a $G$-Schubert variety in ${\bf X}_\Delta$ 
is $\diag(G)$-equivariantly isomorphic to a variety of 
the form $G \times_B X(w)$, where $X(w)$ denotes a Schubert
variety in $\nicefrac{G}{B}$. In particular, this explains
the name {\it $G$-Schubert varieties} as this is the 
name used  for varieties of the form   $G \times_B X(w)$.

In the rest of this section, we will relate $G$-Schubert
varieties to closures of so-called $G$-stable pieces. Our
primary interest are $G$-stable pieces in wonderful 
compactifications but below we will also describe 
the toroidal case in general.

\subsection{$G$-stable pieces in the wonderful compactification}
\label{won}

Let $G=G_{\rm ad}$ denote a group of adjoint type and let $\bf X$
denote its wonderful compactification. Let $J \subset \Delta$ and 
identify ${\bf X}_J$ with $(G \times G) \times_{P_{\D \setminus
J}^- \times P_{\D \setminus J}} {\bf Y}$ as in Section \ref{wonderful}.
Using this identification it easily follows that there exists 
a unique element in ${\bf X}_J$ which is invariant under 
$U_J^- \times U_J$ and $\diag(L_J)$. We denote this element 
by ${\bf h}_J$ and note that as an element of 
 $(G \times G) \times_{P_{\D \setminus
J}^- \times P_{\D \setminus J}} {\bf Y}$
it equals $[(e, e), e_J]$, where $e$ (resp. $e_J$) denotes 
the identity element of $G$ (resp.  the adjoint group 
associated to $L_{\D \setminus J}$). For $w \in W^{\D \setminus J}$, we then let
$${\bf X}_{J, w}=\diag(G) (B w, 1) \cdot {\bf h}_J,$$
and call ${\bf X}_{J, w}$ a $G$-stable piece of ${\bf X}$. 
A $G$-stable piece is a locally closed subset of
${\bf X}$ and by \cite[section 12]{L} and \cite[section 2]{He1}, 
we can use them to decompose ${\bf X}$ as follows
$${\bf
X}=\bigsqcup_{\stackrel{J \subset \D}{ w \in W^{\D \setminus J}}}
{\bf X}_{J, w}.$$
Moreover, by the proof of \cite[Theorem 4.5]{He2}, 
any $G$-Schubert variety is  a finite union
of $G$-stable pieces. In particular, we may
think of $G$-Schubert varieties as closures 
of $G$-stable pieces.

\subsection{$G$-stable pieces in arbitrary toroidal embeddings} 

We fix a toroidal embedding $X$ of $G$. The irreducible components
of the boundary $X \setminus G$ will be denoted by $X_1, \dots, X_n$. 
For each $G \times G$-orbit closure $Y$
in $X$ we then associate the set $$K_Y = \{ i \in \{ 1, \dots, n \}
\mid Y \subset X_i \},$$ where by definition $K_Y = \varnothing$
when $Y = X$. Then by \cite[Prop.6.2.3]{BK}, $Y = \cap_{i \in K_Y}
X_i$. Moreover, we define
$$\ci = \{ K_Y \subset \{ 1, \dots, n \} \mid Y \text{~a
$G \times G$-orbit closure in $X$ } \},$$
and write $X_K := \cap_{i \in K} X_i$ for $K \in \ci$.
Then $(X_K)_{K \in \ci}$ are the set of  closures of $G \times G$-orbits
in $X$. 
Let now $\pi_X : X \rightarrow {\bf X}$ denote the 
given extension of $G \rightarrow G_{\rm ad}$. Then 
the closure of $\pi_X(X_K)$ equals  ${\bf X}_{P(K)}$ for some
unique subset $P(K)$ of $\D$. This defines
a map $P: \ci \rightarrow \mathcal P(\D)$,
where $\mathcal P(\D)$ denotes the set of subsets of $\D$.

As in \cite[5.4]{HT2}, for $K \in \ci$ we may choose a base point $h_K$ in the open $G \times G$-orbit of $X_K$ which maps to  ${\bf h}_{P(K)}$. By \cite[Proposition 5.3]{HT2}, $X_K$ is then naturally isomorphic to $(G \times G) \times_{P^-_{\D \setminus J} \times P_{\D \setminus J}} \overline{L_{\D \setminus J} \cdot h_K}$, where $J=P(K)$ and $\overline{L_{\D \setminus J} \cdot h_K}$ is a toroidal embedding of a quotient $\nicefrac{(L_{\D \setminus J})}{H}$ by some subgroup $H$ of the center of $L_{\D \setminus J}$.

For $K \in \ci$ and $w \in W^{\D \setminus p(K)}$, we then define  
$$X_{K, w}=\diag(G) (B w, 1) \cdot h_K,$$ and call ${X}_{K, w}$ a $G$-stable piece of $X$. One can then show, in the same way as in \cite[4.3]{He2}, that $$X=\bigsqcup_{\stackrel{K \in \ci}{w \in W^{\D \setminus P(K)}}} X_{K, w}.$$
Also similar to the proof of \cite[Theorem 4.5]{He2}, for any $B
\times B$-orbit closure $V$ in $X$, the $G$-Schubert variety
$\diag(G) \cdot V$ is a finite union of $G$-stable pieces.
In particular, $G$-Schubert varieties are closures
of $G$-stable pieces.
 
\section{Frobenius splitting of $G$-Schubert varieties}
\label{sect 9}
In this section, we assume that $X$ is an equivariant 
embedding of $G$. Let $G_{\rm sc}$ denote a
simply connected cover of the semisimple commutator
subgroup $(G,G)$ of $G$. We fix a Borel
subgroup $B_{\rm sc}$ of $G_{\rm sc}$ which is 
compatible with the Borel subgroup $B$ in $G$. 
Similarly we fix a maximal torus $T_{\rm sc} \subset 
B_{\rm sc}$.  

Let $X_1, \dots, X_n$ denote the boundary divisors of $X$. 
The closure within $X$ of the
$B \times B$-orbit $Bs_j w_0 B \subset G$ will
be denoted by $D_j$. Then $D_j$ is  of
codimension 1 in $X$. The translate
$(w_0,w_0) D_j$ of $D_j$ will be denoted
by $\tilde{D}_j$.

By earlier work we know

\begin{thm}\cite[Prop.7.1]{HT2}
\label{can-thm}
The equivariant embedding $X$ admits a $B_{\rm sc} 
\times B_{\rm sc}$-canonical
Frobenius splitting which compatibly Frobenius
splits the closure
of every $B \times B$-orbit.
\end{thm}

As a direct consequence of Theorem \ref{diag canonical}
we then obtain

\begin{cor}
\label{cor-equi}
The variety $(G \times G) \times_{(B \times B)} X$ admits
a $\diag(B_{\rm sc})$-canonical Frobenius splitting 
which is compatible with all subvarieties of the form
$(G \times G)
\times_{(B \times B)} Y$ and ${\rm diag}(G) \times_{{\rm diag}(B)}
Y$, for a $B \times B$-orbit closure $Y$ in $X$.
\end{cor}

\begin{prop}
\label{P-stable} The equivariant embedding $X$ admits a
$\diag(B_{\rm sc})$-canonical Frobenius splitting which compatibly
 splits all 
$G$-Schubert varieties in $X$.

\end{prop}
\begin{proof}
By Corollary \ref{cor-equi} the variety $Z={\rm diag}(G)
\times_{{\rm diag}(B)} X$ admits a $\diag(B_{\rm sc})$-canonical Frobenius
splitting which is compatible with all
subvarieties of the form ${\rm diag}(G) \times_{{\rm diag}(B)}
Y$, with $Y$ denoting a $B \times B$-orbit closure
in $X$. As $X$ is a ${\rm diag}(G)$-stable we may
identify $Z$ with $\nicefrac{G}{B} \times X$ using
the isomorphism 
$$ G \times_B X \rightarrow \nicefrac{G}{B} \times X,$$
$$ [g,x] \mapsto (gB,gx).$$
In particular, we see that the morphism
$$ \pi :  Z = {\rm diag}(G) \times_{{\rm diag}(B)} X
\rightarrow X,$$
$$ [g,x] \mapsto g \cdot x,$$
is projective 
and that $\pi_*(\co_Z) = \co_X$.
As a consequence (see Section \ref{push-forward})
the $\diag(B_{\rm sc})$-canonical Frobenius splitting of 
$Z$ induces a $\diag(B_{\rm sc})$-canonical 
Frobenius splitting of $X$ which is compatible 
with all subvarieties of the form
$$\pi ({\rm diag}(G) \times_{{\rm diag}(B)}
Y)= {\rm diag}(G) \cdot Y,$$
i.e. with all the $G$-Schubert varieties in
$X$. This ends the proof.
\end{proof}

As a direct consequence of Proposition \ref{P-stable}, 
we conclude the following vanishing result 
(see \cite[Theorem 1.2.8]{BK}).

\begin{cor}
\label{cohvan1}
Let $X$ denote a projective equivariant embedding of $G$. Let
$\mathcal X$ denote a $G$-Schubert variety in $X$  and let $\cl$ denote an
ample line bundle on $\mathcal X$. Then $${\rm H}^i(\mathcal X,
\cl)=0, i>0.$$ Moreover, if $\tilde{ \mathcal X}\subset \mathcal X$ is
another $G$-Schubert variety, then the restriction map $${\rm
H}^0(\mathcal X, \cl) \to {\rm H}^0(\tilde{ \mathcal X}, \cl),$$ is
surjective.
\end{cor}

Later (i.e. Cor. \ref{vanishing}) we will generalize the
vanishing part of this result to nef line bundle.

\subsection{F-splittings along ample divisors}
\label{F-split section}

In this subsection we assume that $X$ is toroidal.
The following structural properties
of toroidal embeddings can all be found in
\cite[Sect.6.2]{BK}.
Let $X_0$ denote the complement in $X$ of the union of
the subsets $\overline{B s_i B^-}$ for $i \in \D$.
If we let $\bar{T}$ denote the closure of $T$ in
$X$, then $X_0$ admits a decomposition defined 
by the following isomorphism
\begin{equation}
\label{localiso}
 U \times U^- \times (\bar{T} \cap X_0) \rightarrow X_0, ~
 (x,y,z) \mapsto (x,y) \cdot z.
\end{equation}
Moreover, every $G \times G$-orbit
in $X$ intersects $ (\bar{T} \cap X_0)$ in a 
unique orbit under the left action of $T$. 
Notice here that as $T$ is commutative the 
$T \times T$-orbits and the (left) $T$-orbit 
in $\overline{T}$ will coincide.

\begin{lem}
\label{ample}
Let $X$ denote a projective toroidal equivariant
embedding of $G$ and let $Y$ denote a $G \times
G$-orbit closure in $X$. Let $K$ denote the subset
of $\{ 1, \dots, n \}$ consisting of those $j$ 
such that $Y$ is contained in the boundary component 
$X_j$. Then
$$Y \cap (\bigcup_{j \notin K} X_j \cup \bigcup_{i \in \D}
(1,w_0) D_i),$$
has pure codimension 1 in $Y$ and contains the support of
an ample effective Cartier divisor on $Y$.
\end{lem}
\begin{proof}
Let $X^K= \cup_{j \notin K} X_j$. We claim that
$Y \setminus X^K$ coincides with the open
$G \times G$-orbit $Y_0$ of $Y$. Clearly
$Y_0$ is contained in $Y \setminus X^K$.
On the other hand, let $U$ be a
$G \times G$-orbit in $Y \setminus X^K$.
Then $X_j$ contains $U$ if and only if
$j \notin K$. But every $G \times G$-orbit
closure in $X$ is the intersection of those $X_j$
which contain it \cite[Prop.6.2.3]{BK}.
It follows that the closure of $Y_0$ and $U$
coincide and thus $U = Y_0$.

As $X$ is normal we may choose a
$G \times G$-linearized very ample line bundle $\cl$
on $X$. Then ${\rm H}^0(Y,\cl)$ is a finite 
dimensional (nonzero) representation of 
$G \times G$, and it thus contains a 
nonzero element $v$ which is
$B \times B^-$-invariant up to constants.
The support of $v$ is then the union of
$B \times B^-$-invariant divisors on $Y$.
As $Y_0 \cap (\bar{T} \cap  X_0)$ is a single $T \times T$-orbit
it follows that
$$Y_0 \cap X_0 \simeq U \times U^- \times  (Y_0 \cap (\bar{T} \cap  X_0)),$$
is an affine variety and a single $B \times B^-$-orbit. In particular,
the support of $v$ is contained in
$$ Y \setminus (Y_0 \cap X_0) = Y \cap
(X^K \cup \bigcup_{i \in \D} (1,w_0) D_i).$$
This shows the second part of the statement.
The first part follows as $Y_0 \cap X_0$
is affine \cite[Prop.3.1]{Har}.
\end{proof}

Let now $X$ denote a smooth projective toroidal
embedding of $G$. As the line bundles $\co_X
(D_i)$ and $\co_X (\tilde{D}_i)$ are isomorphic 
it follows by \cite[Prop.6.2.6]{BK} that
\begin{equation}
\label{dua} 
\omega_X^{-1} \simeq \co_X \big(\sum_{i \in \D} (D_i+ \tilde{D}_i)
+ \sum_{j=1}^nX_j\big). 
\end{equation}
Recall that a $X$ is normal and $G$ is semisimple
and simply connected, any line bundle on $X$ will
admit a unique $G_{\rm sc}^2 = G_{\rm sc} \times G_{\rm sc}$-linearization.
In particular, if we let $\tau_i$ denote the canonical 
section of the line bundle $\co_X(D_i)$, then we may 
consider $\tau_i$
as a $B_{\rm sc}^2 = B_{\rm sc} \times B_{\rm sc}$-eigenvector of the space
of global sections of $\co_X(D_i)$. As in the proof 
of \cite[Prop.6.1.11]{BK} we find that the associated
weight of $\tau_i$ equals $\omega_i \boxtimes -w_0 \omega_i$,
where $\omega_i$ denotes the $i$-th fundamental weight.
Similarly, we may consider the canonical section $\sigma_j$
of $\co_X(X_j)$ as a $G_{\rm sc}^2$-invariant
element.

Let $V$ denote a $B \times B$-orbit closure in
$X$. As $V$ is $B \times B$-stable the subset 
$Y = (G \times G) \cdot V$ is closed in $X$. Thus
we may consider $Y$ as the smallest 
$G \times G$-invariant subvariety of $X$ 
containing $V$. Now define $K$ as in Lemma 
\ref{ample} and let $\mathcal M$ denote the
line bundle
$$\cm = \co_X \big((p-1)(\sum_{i \in \D} \tilde{D}_i +
\sum_{j \notin K} X_j )\big).$$
By Equation (\ref{dua}) and  Lemma \ref{duality2}
it then follows that multiplication with $\tau_i^{p-1}$,
for $i \in \Delta$, and $\sigma_j^{p-1}$, for $j \in K$, 
defines a morphism of 
$B_{\rm sc}^2$-linearized line bundles
$$ \mathcal M  \rightarrow \mathcal
End^!_F\big(X,\{ D_i , X_j\}_{i \in \D, j \in K}
\big )\otimes k_{\lambda \boxtimes \lambda},$$
where $\lambda = (1-p)\rho$.
By \cite[Prop.6.5]{HT2} and Lemma \ref{compatible}
any element in $\mathcal End^!_F(X)$
which is compatible with
the closed subvarieties $D_i$, $i \in \D$, and $X_j$, $j \in K$,
is also compatible with $V$ and $Y$.
In particular, we have defined a 
$B_{\rm sc}^2$-equivariant map
\begin{equation}
\label{eta}
\eta : \mathcal M
\rightarrow \mathcal End^!_F(X,Y,V\big )\otimes
k_{\lambda \boxtimes \lambda},
\end{equation}
which, by Lemma  \ref{duality1}, is the same as a 
$B_{\rm sc}^2$-invariant
element $\eta'$ 
in $\End_F^\cm\big(X,Y,V \big) \otimes 
k_{\lambda \boxtimes \lambda}$. In particular,
this defines us an element 
\begin{equation}
\label{v} 
v \in {\rm Ind}_{B_{\rm sc}^2}^{G_{\rm sc}^2}
\big(\End_F^\cm\big(X,Y,V \big) \otimes 
k_{\lambda \boxtimes \lambda} \big),
\end{equation}
which is $G_{\rm sc}^2$-invariant.
We are then ready to use the ideas explained 
in Section \ref{invar}. First we use  
(\ref{invarsect}) to construct a morphism
\begin{equation}
\label{invarsect2}
\End_F^{\cl \boxtimes \cl} \big(
(\nicefrac{G_{\rm sc}}{B_{\rm sc}})^2 \big) \otimes 
\cm(X)  \rightarrow \End_F\big( G_{\rm sc}^2 
\times_{B_{\rm sc}^2}
X \big),
\end{equation}
$$ (u, \sigma) \mapsto \Phi_{\cm,\lambda
\boxtimes \lambda}(u \otimes v \otimes \sigma),$$  
where $\cl$ is the $G_{\rm sc}$-linearized line 
bundle on $\nicefrac{G_{\rm sc}}{B_{\rm sc}}$ associated
to the character $\lambda = (1-p)\rho$. Notice 
that we here have used that $\cm(X)$ is a 
$G_{\rm sc}^2$-module. 


\begin{lem}
\label{lem11}
There exists a $G_{\rm sc}^2$-equivariant
map 
\begin{equation}
\label{global}
 {\rm St} \boxtimes {\rm St}
\rightarrow \cm(X),  
\end{equation}
which maps the $B_{\rm sc}^- \times B_{\rm sc}^-$-invariant
line in ${\rm St} \boxtimes {\rm St}$ to a nonzero  multiple
of the global section 
$$\tilde \sigma = \prod_{i \in \D} \tilde \tau_i^{p-1} \prod_{j \notin K} \sigma_j^{p-1}
\in \cm(X)
,$$ 
where $\tilde{\tau}_i$ denotes the canonical
section of $\co_X(\tilde{D_i})$.
\end{lem}
\begin{proof}
As $\co_X(\tilde{D}_i)$ and $\co_X({D}_i)$ are 
isomorphic as line bundles we may consider the
element 
$$ \sigma = \prod_{i \in \D} \tau_i^{p-1} \prod_{j \notin K} \sigma_j^{p-1}$$
as a global section of $\mathcal M$.
Then $\sigma$ is a $B_{\rm sc}^2$-eigenvector
in $\cm(X)$ of weight $(p-1) \rho \boxtimes (p-1)\rho$.
In particular, $\sigma$ induces a $B_{\rm sc} \times 
B_{\rm sc}$-equivariant map 
$$ k_{(p-1) \rho} \boxtimes k_{(p-1)\rho}
\rightarrow  \cm(X).$$
Applying Frobenius reciprocity and the selfduality
of the Steinberg module $\rm St$, this defines the
desired map
$$ {\rm St} \boxtimes {\rm St} \rightarrow
\cm(X),$$
with the stated properties.
\end{proof}

Combining the map (\ref {invarsect2}) with the map 
(\ref{global}) in Lemma \ref{lem11} we obtain a $G_{\rm sc}^2$-equivariant 
map 
\begin{equation}
\label{invarsect3}
\Theta : \End_F^{\cl \boxtimes \cl} \big(
(\nicefrac{G_{\rm sc}}{B_{\rm sc}})^2 \big) \otimes 
\big( {\rm St} \boxtimes {\rm St} \big) 
  \rightarrow \End_F\big(G_{\rm sc}^2
\times_{B_{\rm sc}^2}
X \big),
\end{equation}
We will now study when the map (\ref{invarsect3})
describes a Frobenius splitting of $G_{\rm sc}^2
\times_{B_{\rm sc}^2} X$. Consider the 
$G_{\rm sc}^2$-equivariant  map
\begin{equation} 
\label{section}
\cm(X) \rightarrow {\rm St} \boxtimes 
{\rm St},
\end{equation} 
$$ \sigma \mapsto \Phi^2_{\cm,\lambda
\boxtimes \lambda}(v \otimes \sigma),$$  
defined as the map (\ref{invarsect1}) in Section 
\ref{invar}. We claim

\begin{lem}
\label{lem2}
The composition of the map (\ref{global})
in Lemma \ref{lem11} and the map in (\ref{section})
is an isomorphism on $  {\rm St} \boxtimes {\rm St}$.
\end{lem}
\begin{proof}
By Frobenius reciprocity it suffices to show that the described 
composed map is nonzero. In particular, 
it suffices to show that 
$$\Phi^2_{\cm,\lambda
\boxtimes \lambda}(v \otimes \tilde \sigma) \neq 0,$$
where $\tilde{\sigma}$ denotes the
global section of $\cm$ defined in 
Lemma \ref{lem11}.
 For this we use the fact that the global section 
$$ \big(\prod_{i \in \D} (\tau_i \tilde{\tau}_i)
\prod_{j =1}^n \sigma_j \big)^{p-1},$$
of $\omega_X^{1-p}$ defines a Frobenius 
splitting of $X$ (see e.g. \cite[proof
of Thm.6.2.7]{BK}). As a consequence 
$\eta(\tilde \sigma)$ is a Frobenius splitting 
of $X$, where $\eta$ is the map
defined in (\ref{eta}). Equivalently , the 
natural $G_{\rm sc}^2$-equivariant morphism 
$$ \End^\cm_F(X) \otimes \cm(X) \rightarrow
k[X'] = k,$$
defined in (\ref{map1}), will map $\eta' \otimes
\tilde{\sigma}$ to $1$. This induces a commutative 
diagram
\begin{equation}
\label{comm5}
\xymatrix{
{\rm Ind}_{B_{\rm sc}^2}^{G_{\rm sc}^2}
\big(\End_F^\cm\big(X \big) \otimes 
k_{\lambda \boxtimes \lambda} \big) \otimes \cm (X) \ar[d]
\ar[rr]^(0.7){\Phi^2_{\cm,\lambda
\boxtimes \lambda}} \ar[rrd]&  & {\rm St} \otimes {\rm St} \ar[d]\\
\End^\cm_F(X) \otimes  k_{\lambda \boxtimes \lambda} \otimes \cm(X) 
\ar[rr]
& & k_{\lambda \boxtimes \lambda}  \\
}
\end{equation}
where the image of $v \otimes \tilde \sigma$ under
the diagonal map is nonzero. This ends the proof. 
\end{proof}

\begin{prop}
\label{propintro}
Let $\Theta$ denote the map defined in
(\ref{invarsect3}). The image $\Theta(\nu)$ 
of an element $\nu$ defines, up to a nonzero
constant, a Frobenius splitting
of $G_{\rm sc}^2 \times_{B_{\rm sc}^2} X$ 
if and only if  the image of $\nu$ under the map
\begin{equation} 
\label{phi2}
\phi_{\lambda \boxtimes \lambda} :
 \End_F^{\cl \boxtimes \cl} \big(
(\nicefrac{G_{\rm sc}}{B_{\rm sc}})^2 \big) \otimes 
\big( {\rm St} \boxtimes {\rm St} \big) \rightarrow k,
\end{equation}
defined in Section \ref{fsplitconclusion}, 
is nonzero. 
\end{prop}
\begin{proof}
Apply Proposition \ref{testFsplit} and Lemma 
\ref{lem2}.

\end{proof}

With the identification $\End_F^{\cl \boxtimes \cl} 
\big((\nicefrac{G_{\rm sc}}{B_{\rm sc}})^2 \big)
\simeq {\rm St} \boxtimes {\rm St}$ the 
map $\phi_{\lambda \boxtimes \lambda}$,
defined in (\ref{phi2}), must necessarily
(up to a nonzero constant) be
the 
$G_{\rm sc}^2$-invariant form on ${\rm St} 
\boxtimes {\rm St}$  mentioned in Section
\ref{61}. Let $v_\Delta$ denote the 
$\diag(G)$-invariant element in 
$\End_F^{\cl \boxtimes \cl} \big(
(\nicefrac{G_{\rm sc}}{B_{\rm sc}})^2 \big)$ 
defined in  Section \ref{well known}.  
Then the $\diag(G)$-equivariant 
map 
$$ {\rm St} \otimes {\rm St} \rightarrow 
k,$$
$$ \nu \mapsto \phi_{\lambda \boxtimes \lambda}(
v_\Delta \otimes \nu),$$
is nonzero and thus it must coincide (up to a 
nonzero constant) with the 
$G_{\rm sc}$-invariant form $\phi$ on 
${\rm St}$ defined in (\ref{form}).

\begin{prop}
\label{ample2}
Fix notation as above and let $D$ denote the
effective Cartier divisor
$$(p-1)\big(\sum_{i \in \D} (1,w_0) {D_i} +
\sum_{j \notin K} X_j \big),$$
on $X$. Then
$X$ admits a Frobenius $D$-splitting which
is compatible with the subvariety
$Y$ and the $G$-Schubert variety 
$\diag(G) \cdot V$.
\end{prop}
\begin{proof}
Consider the $\diag(G)$-equivariant morphism
$$ \Theta_\Delta : {\rm St} \boxtimes {\rm St}
\rightarrow \End_F\big(G_{\rm sc}^2 
\times_{B_{\rm sc}^2} X \big),$$
$$ \nu \mapsto \Theta(v_\Delta \otimes \nu), $$
where  $\Theta$ is the map in (\ref{invarsect3}).
By Lemma \ref{propintro} 
the image $\Theta_\Delta(\nu)$ of an element $\nu \in {\rm St} 
\otimes {\rm St} $ is a Frobenius splitting,
up to a nonzero constant, if and only if 
$\phi(\nu)$ is nonzero. Here $\phi$ is the 
the map defined in (\ref{form}). 

Let $v_+$ (resp. $v_-$) denote a nonzero
$B$ (resp. $B^-$)-eigenvector of $ {\rm St}$
and let $\nu = v_+ \otimes v_-$. After 
possibly multiplying $v_+$ with a constant
we may assume  that $s= \Theta_\Delta (\nu) $ defines a
Frobenius splitting of  
$Z= G_{\rm sc}^2 \times_{B_{\rm sc}^2} X$.  As 
$v$ is compatible with $Y$ and $V$ 
(cf. (\ref{v})) it follows by Theorem
\ref{thm11} and Lemma \ref{diag} 
that $s $ factorizes as
\begin{equation}
\label{decomp2}
s : (F_Z)_* \co_Z \xrightarrow{(F_Z)_* \sigma} (F_Z)_* 
\cm_Z  \xrightarrow{s_1} \co_{Z'},
\end{equation}
where $s_1$ is compatible with the subvarieties
$ G_{\rm sc}^2 \times_{B_{\rm sc}^2} V$, 
$G_{\rm sc}^2 \times_{B_{\rm sc}^2} Y$ and 
$\diag(G_{\rm sc}) \times_{\diag(B_{\rm sc})}
X$. Here $\cm_Z$ is the $G_{\rm sc}^2$-linearized
line bundle on $Z$ associated with the 
$B_{\rm sc}^2$-linearized line bundle $\cm$ on
$X$ as explained in Section \ref{intr},
and $\sigma$ is the global section of $\cm_Z$ 
defined as the image of $\nu$ under the 
map (\ref{global}) in Lemma \ref{lem11}.
Notice that as $\cm$ is a $G_{\rm sc}^2$-linearized 
line bundle on the $G_{\rm sc}^2$-variety $X$ we
may identify the global sections of $\cm$ and
$\cm_Z$. Actually , as $X$ is a $G_{\rm sc}^2$-variety
the morphism
$$ G_{\rm sc}^2 \times_{B_{\rm sc}^2} X 
\rightarrow \nicefrac{G_{\rm sc}}{B_{\rm sc}} \times
\nicefrac{G_{\rm sc}}{B_{\rm sc}} \times X,$$
$$ [(g_1,g_2), x] \mapsto (g_1 B , g_2 B, (g_1,g_2) \cdot x),$$
is an isomorphism. Moreover, under this isomorphism,
the line bundle $\cm_Z$ is just the pull back of 
$\cm$ under projection $p_X$ on the third coordinate.
Thus, by Lemma \ref{lem11} it follows that $\sigma$ 
is the pull back from $X$ of the effective Cartier divisor
$$D= (p-1)\big(\sum_{i \in \D} (1,w_0) {D_i} +
\sum_{j \notin K} X_j).$$
Applying the functor $(p_X)_*$ to (\ref{decomp2}) 
we obtain the Frobenius $D$-splitting 
$$ (p_X)_* s :  (F_X)_* \co_X \xrightarrow{(F_X)_* \sigma_D} (F_X)_* \co(D) 
\xrightarrow{(p_X)_* s_1 } \co_{X'}$$ 
of $X$ where $(p_X)_* s_1$ is compatible with 
the subvarieties $p_X(G_{\rm sc}^2 \times_{B_{\rm sc}^2} Y) = Y$
and $p_X(\diag(G_{\rm sc}) \times_{\diag(B_{\rm sc})} V) = \diag(G) 
\cdot V$ (by Lemma \ref{lem push forward}). This ends the proof.
\end{proof}

\begin{cor}
\label{ample-split}
Let $\mathcal X$ denote a $G$-Schubert variety in a smooth 
projective toroidal embedding of a reductive group $G$. 
Then $\mathcal X$ admits a stable Frobenius splitting along
an ample divisor. 
\end{cor}
\begin{proof}
Apply Proposition \ref{ample2}, Lemma
\ref{ample} and Lemma \ref{sum}.
\end{proof}
\section{Cohomology of line bundles}

The main aim of this section is to obtain a 
 generalizing the vanishing part 
of Corollary \ref{cohvan1} to nef line 
bundles. The concept
of a rational morphism is here central and for this 
we use \cite[Sect.3.3]{BK} 
as a general reference. First we recall :

\begin{defn}
A morphism $f : Y \rightarrow Z$ of varieties 
is a called a {\it rational morphism} if the 
induced map $f^\sharp : \co_Z \rightarrow f_* \co_Y $
is an isomorphism and ${\rm R}^i f_* \co_Y =
0$, $i>0$.
\end{defn}

The following criterion for a morphism to be
rational will be very useful (\cite[Lem.2.11]{Ram}).

\begin{lem}
\label{Kempf}
Let $f : Y \rightarrow Z$ denote a projective morphism
of irreducible varieties and let $\hat Y$ denote a closed
irreducible subvariety of $Y$. Consider the image
$\hat Z = f(\hat Y)$ as a closed subvariety of $Z$. Let
$\mathcal L$ denote an ample line bundle on $Z$
and assume
\begin{enumerate}
\item $f^\sharp : \co_Z \rightarrow f_* \co_Y $
is an isomorphism.
\item ${\rm H}^i(Y, f^* \mathcal L^n ) =
 {\rm H}^i(\hat Y, f^* \mathcal L^n )  = 0$,
for $i>0$ and $n \gg 0$.
\item The restriction map
${\rm H}^0(Y, f^* \mathcal L^n ) \rightarrow
{\rm H}^0(\hat Y, f^* \mathcal L^n )$ is surjective
for $n \gg 0$.
\end{enumerate}
Then the induced map $\hat f : \hat Y \rightarrow \hat Z$
is a rational morphism.
\end{lem}

\subsection{Toric variety}

An equivariant embedding $Z$ of the (reductive) group
$T$ is called a toric variety (wrt. $T$). Notice
that, as $T$ is commutative, we may consider the 
$T \times T$-action on $Z$ as just a $T$-action.
The following result should be well known but, as
we do not know a good reference, we include a 
proof. 

\begin{lem}
\label{toric}
Let $f : Y \rightarrow Z$ denote a 
projective surjective morphism of 
equivariant embeddings of $T$. Let $T \cdot z$
denote a $T$-orbit in $Z$ and let $T \cdot y$ 
denote a $T$-orbit in $f^{-1}(T \cdot z)$ of minimal
dimension. Then the map $T \cdot y \rightarrow 
T \cdot z$, induced by $f$, is an isomorphism.
\end{lem}
\begin{proof}
Let $\overline{T \cdot z}$ and $\overline{T \cdot
y}$ denote the closures of $T \cdot z$ and $
T \cdot y$ in $Z$ and $Y$ respectively. Then 
the induced map 
$$ \hat f : \overline{T \cdot y} \rightarrow 
\overline{T \cdot z},$$ is a projective 
morphism. Moreover, by the minimality assumption
on $T \cdot y$, the inverse image $\hat f^{-1}
(T \cdot z)$ equals $T \cdot y$. In particular,
the induced morphism :   $T \cdot y \rightarrow T \cdot z$
is projective. But any $T$-orbit in a toric 
variety (wrt. to $T$) is isomorphic to a torus 
$T_1$ satisfying that the cokernel of the induced 
map of character groups $X^*(T_1) \rightarrow X^*(T)$ 
is a free abelian group (\cite[Sect.3.1]{Ful}).
In particular, the varieties $T \cdot y$ and 
$T \cdot z$ are tori and the cokernel of 
the induced map of character groups 
$X^*(T \cdot z) \rightarrow X^*(T \cdot y)$ is a
free abelian group. But $T \cdot y 
\rightarrow T \cdot z$ is an affine projective
morphism and thus it must be a finite morphism. 
Thus  the cokernel of $X^*(T \cdot z) \rightarrow 
X^*(T \cdot y)$ is a finite group and, as it is already
a free group, it must be trivial. This ends the 
proof as tori are determined by their character
groups.
\end{proof}

\begin{lem}
\label{rational}
Let $X$ denote a projective embedding of a
reductive group $G$ and let $Y$ denote a
$G \times G$-orbit closure of $X$. Then
there exists a smooth toroidal embedding 
$\hat X$ of $G$, a projective $G$-equivariant
morphism $f : \hat X \rightarrow X$ and
a $G \times G$-orbit closure $\hat Y$ 
in $\hat X$ such that the induced morphism
$f : \hat Y \rightarrow Y$ is a 
rational morphism.

\end{lem}
\begin{proof}
Assume first that $X$ is toroidal.
By \cite[Prop.6.2.5]{BK} there 
exists a smooth toroidal embedding 
$\hat X$ of $G$ with a projective 
morphism $f : \hat X \rightarrow X$. 
Let $X_0$ denote the open subset of
$X$ introduced in the beginning of
Section \ref{F-split section}, and
let $\hat X_0$ denote the 
corresponding subset of $\hat X$.
Then the inverse image $f^{-1}(X_0)$
coincides with $\hat X_0$ 
\cite[Prop.6.2.3(i)]{BK}.
Let $\overline{T}$ (resp. $\hat T$)
denote the closure of $T$ in $X$
(resp. $\hat X$). Then $\overline{T}$
and $\hat T$ are toric varieties 
\cite[Prop.6.2.3]{BK}, and the induced
map $f : \hat T \rightarrow \overline{T}$
is a projective morphism of toric varieties.
Thus also the induced map
$$ \hat X_0 \cap \hat T \rightarrow 
X_0 \cap \overline{T},$$
is a projective morphism of toric varieties.
As mentioned in Section \ref{F-split section}
every $G \times G$-orbit in $X$ will intersect
$X_0 \cap \overline{T}$ in a unique $T$-orbit.
We let $T \cdot x$ denote the open $T$-orbit in 
the intersection 
of $Y$ with $X_0 \cap \overline{T}$. 
By Lemma \ref{toric} we may 
find a $T$-orbit $T \cdot \hat x$ in 
$ \hat X_0 \cap \hat T $ which by $f$ is 
isomorphic to $T \cdot x$, and we then
define $\hat Y$ to be the closure of 
the $G \times G$-orbit through $\hat x$.
By the isomorphism 
(\ref{localiso}) we then conclude that $f$
induces a projective birational morphism
$\hat Y \rightarrow Y$.
By \cite[Cor.8.4]{HT2} the orbit closure
$Y$ is normal and thus, by Zariski's main
theorem, we conclude $f_* \co_{\hat Y}
= \co_Y$. By Lemma \ref{Kempf} (used on
the morphism $\hat Y \rightarrow Y$ and the 
closed non-proper subvariety $\hat Y$ of $\hat Y$) it now
suffices to prove that
$${\rm H}^i(\hat Y, f^*\mathcal L) = 0 ,~i>0,$$
for a very ample line bundle $\cl$ on $Y$.
This follows from \cite[Prop.7.2]{HT2} and
ends the proof in the case when $X$ is
toroidal.

Consider now an arbitrary projective equivariant
embedding $X$ of $G$. Let $\hat X$ denote the
normalization of the closure of the image of
the natural $G \times G$-equivariant embedding
$$  G \rightarrow X \times {\bf X},$$
where ${\bf X}$ denotes the wonderful compactification
of $G_{\rm ad}$. 
Then $\hat X$ is a toroidal embedding of $G$
with an induced projective equivariant
morphism $f : \hat X \rightarrow X$.
Let $\hat Y$ denote any $G \times G$-orbit
closure in $\hat X$. Then $f : \hat Y
\rightarrow f(\hat Y)$ is a rational morphism
\cite[Lem.8.3]{HT2}. In particular, we
may find a $G \times G$-orbit closure
$\hat Y$ of $\hat X$ with
an induced rational morphism
$f : \hat Y \rightarrow Y$. Finally we
may apply the first part of the proof to
$\hat Y$ and $\hat X$  and use that a composition of
rational morphisms is again a  rational
morphism.
\end{proof}

\begin{cor}
\label{vanishing} Let $X$ denote a projective embedding of a
reductive group $G$ and let $\mathcal X$ denote a $G$-Schubert 
variety in $X$. Let $Y= (G \times G) \cdot \mathcal X$
denote the minimal $G \times G$-orbit
closure of $X$ containing $\mathcal X$.
 When $\cl$ is a nef line
bundle on $\mathcal X$ then
$$ {\rm H}^i(\mathcal X, \cl) = 0,~  i>0.$$
Moreover, when $\cl$ is a nef
line bundle on $Y$ then the restriction morphism
$$ {\rm H}^0(Y, \cl) \rightarrow
{\rm H}^0(\mathcal X , \cl),$$
is surjective.
\end{cor}
\begin{proof}
Assume first that $X$ is smooth and toroidal.
Then by Proposition \ref{ample2}, Lemma
\ref{ample} and Lemma \ref{sum} the variety $Y$ admits a stable
Frobenius splitting along an ample divisor
which is compatibly with $\mathcal X$.
Thus the statement follows in this case
by Proposition \ref{prop stable}.

Let now $X$ denote an arbitrary projective equivariant
embedding of $G$. Choose, using Lemma \ref{rational},
a smooth projective toroidal embedding $\hat X$ with 
a projective equivariant morphism $f : \hat X \rightarrow
X$ onto $X$, and a  $G \times G$-orbit closure $\hat Y$ in 
$\hat X$ with an induced rational morphism onto $Y$. Let $V$ denote a
$B \times B$-orbit closure  in $Y$ such that
$\mathcal X = \diag(G) \cdot {V}$. As $Y$ is
the minimal $G \times G$-orbit closure containing
$\mathcal X$ it follows that $V$ will intersect
the open $G \times G$-orbit of $Y$. In
particular, there exists a $B \times B$-orbit
closure $\hat V$ in $\hat X$ which intersects 
the open $G \times G$-orbit of $\hat Y$ and 
which maps onto $V$. In particular,
$$\hat{\mathcal X} := \diag(G) \cdot
{\hat V},$$ 
is a $G$-Schubert variety in $\hat X$ which by $f$
maps onto $\mathcal X$. Moreover, $\hat Y$ is the 
minimal $G \times G$-orbit closure containing 
$\hat{\mathcal X} $. 

We claim that the induced morphism $\hat{\mathcal X}
\rightarrow \mathcal X$ is a rational morphism. To 
prove this we apply Lemma \ref{Kempf} to the rational
morphism $f : \hat Y \rightarrow Y$. Choose an ample
line bundle $\cm$ on $Y$. Then it suffices to prove 
that  
\begin{equation}
\label{cond1}
{\rm H}^i(\hat Y, f^* \cm^n ) =
 {\rm H}^i(\hat{\mathcal X}, f^* \cm^n )  = 0, 
~i>0, ~n>0,
\end{equation} 
and that the restriction map 
\begin{equation}
\label{cond2}
{\rm H}^0(\hat Y, f^* \cm^n ) \rightarrow
{\rm H}^0(\hat X, f^* \cm^n ),
\end{equation} 
is surjective for $n > 0$. But $\cm^n$ is an
ample, and thus nef, line bundle on $Y$ and 
therefore the pull back $f^* \cm^n$ is a
nef line bundle on $\hat Y$ 
(\cite[Ex. 1.4.4]{Laz}). As $\hat X$ is smooth
and toroidal, the conclusion of the first 
part of this proof
then shows that  conditions (\ref{cond1}) and
(\ref{cond2}) are satisfied.

Now both $\hat{\mathcal X} \rightarrow 
\mathcal X$ and $\hat Y \rightarrow Y$ are 
rational morphisms. In particular, we have
identifications 
$$ {\rm H}^i(\hat Y, f^* \cl ) \simeq
 {\rm H}^i(Y, \cl ), 
~i \geq 0,$$
$$  {\rm H}^i(\hat{\mathcal X}, f^* \cl ) \simeq
 {\rm H}^i(\mathcal X, \cl ), 
~i \geq 0,$$
for any line bundle $\cl$ on $Y$ or, in the 
second equation, on $X$. When $\cl$ is
a nef line bundle the pull back $f^* \cl$ is also
nef (\cite[Ex. 1.4.4]{Laz}). Thus as we have already 
completed the proof of the statement for smooth 
toroidal embeddings, in particular for $\hat X$,
this now ends the proof. 
\end{proof}

By the proof of the above result we also find that any $G$-Schubert 
variety $\mathcal X$ in a projective equivariant embedding of $G$,
will admit a $G$-equivariant rational morphism $f : \hat{\mathcal X}
\rightarrow {\mathcal X}$ by a $G$-Schubert variety $\hat{\mathcal X}$
of some smooth projective toroidal embedding of
$G$.

\begin{remark}
When $X = {\bf X}$ is the wonderful compactification
of a group $G$ of adjoint type and $\cl$ is a nef
line bundle on ${\bf X}$, then the restriction
morphism
$$ {\rm H}^0({\bf X}, \cl) \rightarrow {\rm H}^0(Y, \cl), $$
to any closed $G \times G$-stable irreducible subvariety
$Y$ of $\bf X$ is surjective. In particular, also
the restriction morphism
$$ {\rm H}^0({\bf X}, \cl) \rightarrow {\rm H}^0(\mathcal X, \cl), $$
to any $G$-Schubert variety $\mathcal X$ is surjective by the above
result. We do not know if the latter is true for arbitrary
equivariant embeddings.
\end{remark}

\section{Normality questions}
The obtained Frobenius splitting properties of
$G$-Schubert varieties in Section 
\ref{sect 9} and the cohomology 
vanishing results in Corollary \ref{vanishing} 
should be expected to have strong implications
on the geometry of these varieties. However, in this 
section we provide an example of a $G$-Schubert
variety in the wonderful compactification of
a group of type $G_2$ which is not even normal. 
In fact, it seems that there are plenty of such 
examples.

\subsection{Some general theory}
We keep the notations as  in Section \ref{won}.
For $J \subset \D$ and  $w \in W^{\D \setminus J}$,
we let $\overline{{\bf X}_{J, w}} $ denote the closure of
${\bf X}_{J, w}$ in $\bf X$.  Let
$$K=\max\{K' \subset \D \setminus J; w K' \subset K' \}.$$
By \cite[Prop. 1.12]{He2}, we have a $\diag(G)$-equivariant isomorphism
$$
 {\diag(G)} \times_{\diag(P_K)} (P_K \dot w,P_K) {\bf h}_J
\simeq {\bf X}_{J, w}
$$
induced by the inclusion of $ (P_K \dot w,P_K) {\bf h}_J$
in ${\bf X}$.
Let $V$ denote the closure of $(P_K \dot w,P_K) {\bf h}_J$
within ${\bf X}$. Then $V$ is the closure of a
$B \times B$-orbit and we find that the induced
map
\begin{equation}
\label{birat} 
f : {\diag(G)} \times_{\diag(P_K)} V
\rightarrow  \overline{{\bf X}_{J, w}},
\end{equation}
is a birational and projective morphism. Thus, 
by Zariski's Main Theorem, a necessary condition
for $\overline{{\bf X}_{J, w}}$  to be normal is 
that the fibers of $f$ are connected. Actually, 
in positive characteristic, connectedness of 
the fibers is also sufficient for 
$\overline{{\bf X}_{J, w}}$ to be normal. This 
follows as $\overline{{\bf X}_{J, w}}$ is 
Frobenius split (Prop. \ref{P-stable}) and thus weakly normal 
\cite[Prop.1.2.5]{BK}.

\subsection{An example of a non-normal closure}

Let now, furthermore, $G$ be a group of type $G_2$. Let
$\alpha_1$ denote the short simple root and $\alpha_2$
denote the long simple root. The associated simple
reflections are denoted by $s_1$ and $s_2$. Let
$J = \{ \alpha_2 \}$ 
and  $w = s_1 s_2 \in W^{\D \setminus J}$. In this case $K =
\emptyset $ and we obtain a birational map
$$ f : {\diag(G)} \times_{\diag(B)} V
\simeq \overline{{\bf X}_{J, w}}$$
where $V$ is the closure of $(B\dot w,B) {\bf h}_J$. By \cite[Prop. 2.4]{Sp}, the part of $V$ which intersect
the open $G \times G$-orbit of ${\bf X}_J$ equals
\begin{equation}
\label{a} 
\bigcup_{w \leq w'} (B \dot w',B) {\bf h}_J
\cup \bigcup_{w s_1 \leq w'} (B \dot w',B \dot s_1) {\bf h}_J.
\end{equation}
In particular, $x:=(\dot v, 1) {\bf h}_J$ is an element of $V$, where $v=s_2 s_1 s_2$. We claim that the fiber of $f$ over $x$
is not connected. To see this let $y$ denote
a point in the fiber over $x$. Then we may
find $g \in G$ and $\tilde{x}
\in V$ such that
$$ y = [g, \tilde{x}].$$
By (\ref{a}), $\tilde{x} = (b \dot w', b') {\bf h}_J$ for some $b \in B$, $b' \in P_{\D \setminus J}$
and $w' \geq w$.
Then
$$( g b \dot w', g b') {\bf h}_J
=(\dot v,1) {\bf h}_J.$$
It follows that $(\dot{v}^{-1}
g b \dot w', g b')$ lies in
the stabilizer of ${\bf h}_J$. In particular,
$ g b' \in P_{\D \setminus J}$ and thus also $g \in P_{\D \setminus J}$.
If $g \in B$ then $y = [1, x]$. So assume
that $g = u_1(t) \dot s_1$ where $u_1$ is
the root homomorphism associated to
$\alpha_1$. Assume that $t \neq 0$. Then
we may find $b_1 \in B$ and $s \in k$
such that
$g = u_{-1}(s) b_1$
where $u_{-1}$ is the root homomorphism
associated to $-\alpha_1$. Thus
\begin{equation}
\notag
\begin{split}
\tilde{x} & = (g^{-1}, g^{-1})
(\dot v,1) {\bf h}_J \\
& = (b_1^{-1} u_{-1}(-s) \dot v, g \i) {\bf h}_J \\
& = (b_1^{-1} \dot v, g \i) {\bf h}_J \\
& \in (B \dot v, B \dot s_1) {\bf h}_J
\end{split}
\end{equation}
where the third equality follows as
$\dot v^{-1} u_{-1}(-s) \dot v$ is contained
in the unipotent radical of $P_{\D \setminus J}^-$. But
$(B \dot v, B \dot s_1) {\bf h}_J $ has empty intersection
with $V$ (by (\ref{a})) which contradicts the assumption
that $t \neq 0$. It follows that the only
possibilities for $y$ are $[1,x]$ and
$[\dot s_1 , (\dot s_1^{-1} \dot v, \dot s_1^{-1}) {\bf h}_J]$.
As $ (\dot s_1^{-1} \dot v, \dot s_1^{-1})$ is contained 
in $V$  (by (\ref{a}))
we conclude that the fiber of $f$ over
$x$ consists of 2 points; in particular
the fiber is not connected and thus $\overline{{\bf X}_{J, w}}$
is not normal.

\begin{remark}
It seems likely that normalizations of $G$-Schubert 
varieties should have nice singularities : If we
let $\mathcal Z_{J, w}$ denote the normalization
of the closure of ${\bf X}_{J, w}$, then the map 
(\ref{birat}) induces a birational and projective
morphism 
$${\tilde f} :   {\diag(G)} \times_{\diag(P_K)} V
\rightarrow \mathcal Z_{J, w}.$$
We expect that $\tilde f$ can be used to obtain
global $F$-regularity of $\mathcal Z_{J, w}$
(see  \cite{S} for an introduction
to global $F$-regularity). In fact, by the 
results in \cite{HT2} the $B \times B$-orbit 
closure $V$ is globally $F$-regular.  Thus
${\diag(G)} \times_{\diag(P_K)} V$
is locally strongly $F$-regular, and as 
$$\tilde{f}_* \co_{ {\diag(G)} \times_{\diag(P_K)} V}
= \co_{\mathcal Z_{J, w}},$$ 
it seems likely that $\mathcal  Z_{J, w}$
is also locally  strongly $F$-regular.
Moreover, similarly to Corollary
\ref{ample-split}
one may conclude that
$\mathcal Z_{J, w}$ admits a stable Frobenius
splitting along an ample divisor.
Thus $\mathcal Z_{J, w}$ is globally $F$-regular if it is
locally strongly $F$-regular.
At the moment we do
not know if $\mathcal  Z_{J, w}$ is
locally strongly $F$-regular. 
\end{remark}

\section{Generalizations}
\label{Rstable}
Fix notation as in Section \ref{notation}. 
An admissible triple of $G \times
G$ is by definition a triple $\cc=(J_1, J_2,
\th_{\d})$ consisting of $J_1, J_2 \subset \D$, a
bijection $\d: J_1 \rightarrow J_2$
and an isomorphism $\th_{\d}: L_{J_1}
\rightarrow L_{J_2}$ that maps $T$ to $T$ and
the root subgroup $U_{\a_i}$ to the root subgroup $U_{\a_{\d(i)}}$
for $i \in J_1$.
To each admissible triple $\cc=(J_1, J_2, \th_{\d})$,
we associate the subgroup $\car_{\cc}$ of $G \times G$ defined
by
$$
\car_{\cc}=\{ (p, q) : p \in P_{J_1}, q \in P_{J_2}, \th_{\d}(\pi_{J_1}(p))=
\pi_{J_2}(q) \},
$$
where $\pi_J : P_J \rightarrow L_J$, for a subset $J \subset \D$,
denotes the natural quotient map. 

Let $X$ denote an equivariant embedding of the reductive
group $G$. A $\car_{\cc}$-Schubert variety of $X$ is then 
a subset of the form $\car_{\cc} \cdot V$ for some 
$B \times B$-orbit closure $V$ in $X$. When $G=G_{\rm ad}$ 
is a group of adjoint type and $X = {\bf X}$ is the 
associated wonderful compactification the set of 
$\car_{\cc}$-Schubert varieties coincides with closures 
of the set of $\car_{\cc}$-stable pieces. 
By definition  \cite[section 7]{LY}, 
a $\car_{\cc}$-stable piece in the wonderful compactification 
${\bf X}$ of $G_{\rm ad}$ is a subvariety of the form $\car_{\cc} \cdot Y$, 
where $Y=(B v_1, B v_2) \cdot {\bf h}_J$ for some $J \subset \D$, $v_1 
\in W^J$ and $v_2 \in {}^{J_2} W$ (notation as in Section \ref{won}).  
Notice that when $J_1=J_2=\D$ and $\th_\d$ is the identity map
 then a $\car_{\cc}$-stable piece is the same as a $G$-stable 
piece. On the other hand, when $J_1=J_2=\emptyset$, then a
 $\car_{\cc}$-stable piece is the same as a $B \times B$-orbit.
 Moreover, any $\car_{\cc}$-Schubert variety is a finite 
union of   $\car_{\cc}$-stable pieces  \cite[Section 7]{LY}.

The following is a generalization of Proposition
\ref{P-stable} and Proposition \ref{ample2}.

\begin{prop}
\label{stable}
Let $\cc=(J_1, J_2,\th_{\d})$ denote an admissible 
triple of $G \times G$ and let $X$ denote an 
equivariant embedding of $G$. Then $X$ admits
a Frobenius splitting which compatible splits
all $\car_{\cc}$-Schubert varieties in $X$.  
If, moreover, $X$ is a smooth, projective and 
toroidal embedding and $Y=X_K = (G \times G) \cdot V$, 
for some $B \times B$-orbit closure $V$ in $X$,
then $X$ 
admits a Frobenius splitting along the 
Cartier divisor
$$D=(p-1) \big(\sum_{i \in \D}(w_0^{J_1},1)\tilde{D}_i+\sum_{j \notin K} X_j \big),$$
which is compatibly with $Y$ and $\car_{\cc} \cdot V$.
\end{prop}
\begin{proof}
As the proof is similar to the proof of 
Proposition \ref{P-stable} and Proposition 
\ref{ample2} we only sketch the 
proof. In the following $G_J$, for a subset 
$J \subset \D$,  denotes the commutator  of the Levi subgroup 
in $G_{\rm sc}$ associated to $J$. The Borel subgroup 
$G_J \cap B_{\rm sc}$ of $G_J$ is denoted by $B_J$.
Define $X_\cc$
to be the  $G_{J_1}^2$-variety
which as a variety is $X$ but where
the action is twisted by the morphism
$$ G_{J_1} \times G_{J_1} \xrightarrow
{{\bf 1} \times \th_{\d}}
G_{J_1} \times G_{J_2}.$$
Then the  $B_{J_1} \times B_{J_2}$-canonical
Frobenius splitting of $X$ defined by
Theorem \ref{can-thm} and Lemma
\ref{can restriction} induces a
 $B_{J_1}^2$-canonical
Frobenius splitting of $X_\cc$.
In particular, all subvarieties of
$X_\cc$ which corresponds to $B
\times B$-orbit closures in $X$
will be compatibly Frobenius split
by this canonical Frobenius splitting.
Now apply an argument as in the
proof of Proposition
\ref{P-stable} and use the identification
of $\car_{\cc} \cdot V  \subset X$
with ${\rm diag}(G_{J_1})
\cdot V \subset X_\cc$. This
ends the proof of the first
statement.

Assume now that $X$ is a smooth,
projective and toroidal embedding
and consider the $B_{\rm sc}^2$-equivariant 
morphism 
$$ \eta : \mathcal M \rightarrow \mathcal End^!_F(X,Y,V) \otimes
k_{(1-p)\rho \boxtimes (1-p) \rho},$$
defined in (\ref{eta}). Let $Y_\cc$ and $V_\cc$ 
be defined similar to $X_\cc$. Then $\eta$ induces
a $B_{J_1}^2$-equivariant
morphism
$$ \eta_\cc : \mathcal M \rightarrow \mathcal End^!_F(X_\cc,Y_\cc ,V_\cc) \otimes
k_{(1-p)\rho_{J_1} \boxtimes (1-p) \rho_{J_1}}.$$
Similar to the definition of $v$ in (\ref{v}) we
obtain from $\eta_\cc$ an element
$$ v_\cc \in {\rm Ind}_{B_{J_1}^2}^{G_{J_1}^2}
\big(
\End_F(X_\cc,Y_\cc ,V_\cc) \otimes
k_{(1-p)\rho_{J_1} \boxtimes (1-p) \rho_{J_1}}
\big),$$
and from this a $G_{J_1}^2$-equivariant morphism 
\begin{equation}
\label{vcc}
\End_F^{\cl_{J_1} \boxtimes \cl_{J_1}} \big(
(\nicefrac{G_{J_1}}{B_{J_1}})^2 \big) \otimes 
\cm(X_\cc)  \rightarrow \End_F\big( G_{J_1}^2 
\times_{B_{J_1}^2}
X_\cc \big),
\end{equation}
similar to (\ref{invarsect2}). Here 
$\cl_{J_1}$ is the line bundle on $\nicefrac{G_{J_1}}{B_{J_1}}$
associated to the character $ (1-p) \rho_{J_1}$. Combining
Lemma \ref{can restriction} and Lemma \ref{lem11}
we also obtain a map 
\begin{equation}
\label{lem11cc}
{\rm St_{J_1}} \boxtimes {\rm St_{J_1}}
\rightarrow \cm(X_\cc),
\end{equation}
with properties similar to the ones described in 
 Lemma \ref{lem11}. As in (\ref{section}) we may 
also use $v_\cc$ to construct a morphism
$$ \cm(X_\cc) \rightarrow 
{\rm St_{J_1}} \boxtimes {\rm St_{J_1}},$$
such that the composition with (\ref{lem11cc})
is an isomorphism on 
${\rm St_{J_1}} \boxtimes {\rm St_{J_1}}$. Finally we may construct 
\begin{equation}
\notag
\Theta_{\cc} : \End_F^{\cl_{J_1} \boxtimes \cl_{J_1}} \big(
(\nicefrac{G_{J_1}}{B_{J_1}})^2 \big) \otimes 
({\rm St_{J_1}} \boxtimes {\rm St_{J_1}})  
\rightarrow \End_F\big(G_{J_1}^2
\times_{B_{J_1}^2}
X_\cc \big),
\end{equation}
similar to (\ref{invarsect3}). In particular, 
a statement equivalent to 
Proposition \ref{propintro} is satisfied 
for $\Theta_\cc$. Let $v_{+}^{J_1}$ (resp. 
$v_-^{J_1}$) denote a highest (resp. lowest)
weight vector in ${\rm St}_{J_1}$ and let 
$v_\Delta^{J_1}$ denote the $\diag(G_{J_1})$-invariant 
element  of  $\End_F^{\cl_{J_1} \boxtimes \cl_{J_1}} \big(
(\nicefrac{G_{J_1}}{B_{J_1}})^2 \big) $. Imitating 
the proof of Proposition \ref{ample2} we then 
find that $\Theta_\cc(v_\Delta^{J_1} \otimes 
(v_{+}^{J_1} \otimes v_-^{J_1}))$ is a Frobenius
splitting of $G_{J_1}^2 \times_{B_{J_1}^2}
X_\cc $ (up to a nonzero constant). Moreover, 
the push forward of this Frobenius splitting
to $X$ has the desired properties. We only 
have to note that the  effective Cartier 
associated to the image of $v_{+}^{J_1} \otimes
v_-^{J_1}$ under the map (\ref{lem11cc}) equals
$$D=(p-1) \big(\sum_{i \in \D}(w_0^{J_1},1)\tilde{D}_i+\sum_{j \notin K} X_j \big).$$
This ends the proof.
\end{proof}

\

We may also argue as in Corollary \ref{vanishing}
to obtain

\begin{cor}
\label{vanishing2}
Let $X$ denote a projective embedding of a
reductive group $G$ and let $V$ denote the
closure of a $B \times B$-orbit in $X$. Let
$Y= ( G \times G) \cdot V$ and 
$\mathcal X_\cc = \car_\cc \cdot V$.
When $\cl$ is a nef line bundle on
$\mathcal X_\cc$ then
$$ {\rm H}^i(\mathcal X_\cc, \cl) = 0,~  i>0.$$
Moreover, when $\cl$ is a nef
line bundle on $Y$ then the restriction morphism
$$ {\rm H}^0(Y, \cl) \rightarrow
{\rm H}^0(\mathcal X_\cc , \cl),$$
is surjective.
\end{cor}

\begin{remark}
In the case where $k=\mathbb C$ and $X$ is the wonderful
compactification, the subvarieties $(w_0^{J_1}, 1) \tilde D_i$,
$X_j$ and all the $\car_{\cc}$-Schubert varieties are Poisson
subvarieties with respect to the Poisson structure on $X$
corresponding to the splitting
$$\Lie(G) \oplus \Lie(G)=l_1 \oplus l_2,$$ where
$l_1=\Lie(\car_{\cc})$ and $l_2$ is a certain subalgebra of
$\Ad(w_0^{J_1}) \Lie(B^-) \oplus \Lie(B^-)$. See \cite[4.5]{LY2}.

\end{remark}

\bibliographystyle{amsalpha}

\end{document}